\documentclass[12pt]{amsart}

\usepackage{framed}
\usepackage{braket}
\usepackage{amsmath,amssymb,amsthm}
\usepackage{mathtools}
\usepackage{color}
\usepackage{bm}
\usepackage[all]{xy}
\usepackage{amscd}
\usepackage{graphicx}
\usepackage{ascmac}
\usepackage{BOONDOX-frak, mathrsfs}
\usepackage[top=30truemm,bottom=30truemm,left=25truemm,right=25truemm]{geometry}
\usepackage{mathtools}
\usepackage{stmaryrd}
\mathtoolsset{showonlyrefs=true}
\makeatletter
\@addtoreset{equation}{section}

\makeatother

\newcommand{\Z}{\mathbb{Z}}
\newcommand{\Q}{\mathbb{Q}}
\newcommand{\C}{\mathbb{C}}

\newcommand{\R}{\mathbb{R}}
\newcommand{\trivial}{\bf{1}}

\newcommand{\OO}{\mathcal{O}}
\newcommand{\GG}{\mathcal{G}}

\newcommand{\RG}{{\rm R}\Gamma}

\newcommand{\ff}{\mathfrak{f}}
\newcommand{\cK}{\mathcal{K}}

\newcommand{\cL}{\mathcal{L}}

\newcommand{\ol}[1]{\overline{#1}}
\newcommand{\parenth}[1]{\left( #1 \right)}

\newcommand{\otimesL}{\otimes^{\mathbb{L}}}

\DeclareMathOperator{\Gal}{Gal}

\DeclareMathOperator{\Imag}{Im}
\DeclareMathOperator{\Fitt}{Fitt}

\DeclareMathOperator{\ram}{ram}

\DeclareMathOperator{\ord}{ord}
\DeclareMathOperator{\Hom}{Hom}
\DeclareMathOperator{\RHom}{RHom}
\DeclareMathOperator{\rank}{rank}

\DeclareMathOperator{\ab}{ab}
\DeclareMathOperator{\Det}{Det}

\DeclareMathOperator{\cyc}{cyc}

\DeclareMathOperator{\perf}{perf}
\DeclareMathOperator{\loc}{loc}
\DeclareMathOperator{\Real}{Re}
\DeclareMathOperator{\Twist}{Tw}
\DeclareMathOperator{\twist}{tw}
\DeclareMathOperator{\Cone}{Cone}

\DeclareMathOperator{\Iw}{Iw}
\DeclareMathOperator{\id}{id}
\DeclareMathOperator{\res}{res}
\DeclareMathOperator{\reg}{reg}

\DeclareMathOperator{\Aut}{Aut}
\DeclareMathOperator{\cores}{cores}

\DeclareMathOperator{\inv}{inv}
\DeclareMathOperator{\Spec}{Spec}

\DeclareMathOperator{\rk}{rk}
\DeclareMathOperator{\ev}{ev}
\DeclareMathOperator{\D}{D}
\let\oldenumerate\enumerate
\renewcommand{\enumerate}{
   \oldenumerate
   \setlength{\itemsep}{1pt}
   \setlength{\parskip}{0pt}
   \setlength{\parsep}{0pt}
}
\let\olditemize\itemize
\renewcommand{\itemize}{
   \olditemize
   \setlength{\itemsep}{1pt}
   \setlength{\parskip}{0pt}
   \setlength{\parsep}{0pt}
}

\theoremstyle{plain}
\newtheorem{thm}{Theorem}[section]
\newtheorem{lem}[thm]{Lemma}

\newtheorem{prop}[thm]{Proposition}
\newtheorem{cor}[thm]{Corollary}

\theoremstyle{definition}
\newtheorem{defn}[thm]{Definition}

\newtheorem{rem}[thm]{Remark}

\title[On zeta elements and functional equations]
{On zeta elements and functional equations for Tate motives over totally real fields}

\author{Mahiro Atsuta}
\address{Department of Mathematics, Faculty of Science and Technology, Tokyo University of Science. 
2641 Yamazaki, Noda City, Chiba 278-8510, Japan}
\email{mahiro\_atsuta@rs.tus.ac.jp}


\makeatletter
\@namedef{subjclassname@2020}{\textup{2020} Mathematics Subject Classification}
\makeatother

\date{\today}

\begin{document}
\begin{abstract}
In this paper, we study Iwasawa theory for Tate motives over totally real fields. 
More precisely, we construct a zeta element that interpolates the values of $L$-functions at positive integers 
over totally real fields under a certain unramified condition at $p$. 
As an application of this, 
we construct a canonical element in the exterior power bidual of the Galois cohomology group 
that is also related to the values of $L$-functions at positive integers. 
%a higher rank Euler system over a totally real field that is related to the values of $L$-functions at positive integers.
\end{abstract}
\maketitle

%\tableofcontents

%%%%%%%%%%%%%%%%%%%%%
\section{Introduction}\label{Intro}
%%%%%%%%%%%%%%%%%%%%%
One of the central themes in number theory is the study of relationships between arithmetic objects and the special values of $L$-functions. 
The equivariant Tamagawa number conjecture (ETNC), formulated by Burns--Flach \cite{BuFl01}, is a very general and powerful conjecture describing such relationships. 
More concretely, the ETNC predicts that the determinant module of an associated global arithmetic complex admits a canonical basis, called a ``zeta element'', which is related to the special values of $L$-functions. 

On the other hand, the generalized Iwasawa main conjecture (GIMC), formulated by Kato \cite{Kato91} and Fukaya--Kato \cite{FK06}, predicts a statement similar to that of the ETNC.
The GIMC is stronger than the ETNC in that it also incorporates the compatibility of zeta elements under deformations of motives.

In this paper, we focus on Tate motives over a totally real field and study these conjectures in this setting. 
Let $p$ be an odd prime and $j$ be a positive integer. 
The first main result of this paper is to construct a zeta element for $\Z_p(j)$, which interpolates the values of $L$-functions at $s = j$ over totally real fields under a certain unramified condition at $p$ (Theorem \ref{thm:zeta>0}(1) below). 
Moreover, we prove that these zeta elements are compatible with deformations of the cyclotomic twist (Theorem \ref{thm:zeta>0}(2)). 
These results provide a partial verification of the generalized Iwasawa main conjecture in this case of Tate motives.

As an application of this result, we construct a canonical element in the exterior power bidual of the Galois cohomology group of $\Z_p(1-j)$ for each positive integer $j$. 
The image of this element under the Bloch--Kato dual exponential map coincides with the value of the corresponding $L$-function at $s=j$ (Theorem \ref{thm:ES}(ii) below).
The existence of such elements was predicted by Burns--Kurihara--Sano \cite{BKS20}, who call them a ``generalized Stark element". 
Furthermore, the resulting collection of these elements forms a (higher rank) Euler system (see Remark \ref{rem:ES} below).
We explain this in more detail later. 

This paper is inspired by recent work of Burns--Sano \cite{BS20functional}.
They established the local (non-equivariant) Tamagawa number conjecture for $\Z_p (j)$ when $p$ is unramified in the base field \cite[Theorem 2.3]{BS20functional}. 
Moreover, by combining this with the Iwasawa main conjecture, they constructed a higher rank Euler system over a totally real field \cite[Theorem 5.2]{BS20functional}. 
A basic strategy for proving the main results in this paper follows a similar approach to this work. 

Finally, it is worth mentioning that Sakamoto \cite{Sak23} independently constructed higher rank Euler systems, distinct from the work of Burns--Sano.
Although the construction in \cite{Sak23} also relies on the Iwasawa main conjecture, it is notable for its use of a non-canonical ``rank-reduction'' technique for Euler systems.

%%%%%%%%%%%%%%%%%%%%%%%%%%%%%%%%%%%%%
\subsection{Main results}\label{ss:main1}
%%%%%%%%%%%%%%%%%%%%%%%%%%%%%%%%%%%%%%%
We state the main results of this paper here. 

Let $K/k$ be a finite abelian CM-extension; i.e., 
$K/k$ is a finite abelian extension of number fields such that 
$k$ is a totally real field and $K$ is a CM-field. 
We fix an odd prime number $p$ and consider a cyclotomic extension $K_\infty := \cup_{n\geq0} K (\mu_{p^n})$, 
where $\mu_{p^n}$ is the set of $p^n$-th roots of unity. 
We set $\Lambda := \Z_p[[\Gal (K_\infty / k)]]$ and 
\[
\Omega (K_\infty / k) := \{ k \subset L \subset K_\infty \; | \; [L : k] < \infty \text{ and } L \text{ is a CM-field} \}. 
\]
For any integer $j$, consider an idempotent 
\[
\varepsilon_j := \frac{1+(-1)^j c}{2} \in \Lambda, 
\]
where $c \in \Gal (K_\infty / k)$ is the complex conjugation. 
For any $L \in \Omega (K_\infty / k)$, we put $\GG_L := \Gal (L/k)$ and 
we also write $\varepsilon_j \in \Z_p[\GG_L]$ for the idempotent which is defined similarly. 
We also put 
$X_L (j) := \bigoplus_{\iota : L \hookrightarrow \C} \Z_p (j),  
$
where $\iota$ runs over the set of embeddings, and we set $X_L (j)^+ := \varepsilon_j X_L (j)$. 
%Then, the complex conjugation $c_\R \in \Gal (\C / \R)$ acts on this 
%by $c_\R \cdot (x_\iota)_{\iota : M \hookrightarrow \C} = ((-1)x_\iota)_{c_\R \circ}$. 

Let $S$ be a finite set of places of $k$ such that $S$ contains $S_p \cup S_\infty \cup S_{\ram} (K/k)$, 
where $S_p$ is the set of all $p$-adic primes, $S_\infty$ is the set of all archimedean places, and $S_{\ram} (K/k)$ is the set of primes which ramify in $K/k$. 
Then, for any $L \in \Omega (K_\infty / k)$ and positive integer $j$, 
we define the graded invertible $\Z_p[\GG_L]$-module 
\[
\Xi_{L/k, S} (j) := \Det_{\Z_p[\GG_L]}^{-1} (\RG (G_{L, S} , \Z_p (1-j))) \otimes_{\Z_p[\GG_L]}
\Det_{\Z_p[\GG_L]}^{-1} (X_L (-j)^+).  
\]
Then, we have an isomorphism
\[
\vartheta_{L / k, S}^j : \C_p \otimes_{\Z_p} \varepsilon_j\Xi_{L/k, S} (j) \xrightarrow{\sim} \varepsilon_j \C_p [\GG_L]
\]
for any $L \in \Omega (K_\infty / k)$ and positive integer $j$ (see Definition \ref{def:period_reg>0} below). 
Furthermore, we define $\Xi_{K_\infty/k, S} (j)$ as the projective limit of $\Xi_{L/k, S} (j)$ for $L \in \Omega(K_\infty / k)$. 
We put $r_k := [k : \Q]$. 

\begin{thm}\label{thm:zeta>0}
Assume that $k / \Q$ is unramified at $p$ and $K/k$ is also unramified at all $p$-adic primes. 
Then, for any $j \in \Z_{\geq 1}$, 
there is a (unique) $\varepsilon_j \Lambda$-basis 
\[
\varepsilon_j Z_{K_{\infty} /k, S}^{j} \in \varepsilon_j \Xi_{K_\infty / k, S} (j) 
\]
such that the following statements hold. 

\begin{itemize}
\item[(1)]
For any $L \in \Omega (K_\infty / k)$, 
the following composite map 
\[
\varepsilon_j \Xi_{K_\infty / k, S} (j) \twoheadrightarrow \varepsilon_j \Xi_{L / k, S} (j)
\xrightarrow{\vartheta_{L / k , S}^j} \varepsilon_j \C_p[\GG_{L}] 
\]
sends $\varepsilon_j Z_{K_\infty/k, S}^j$ to 
\[
(-1)^{r_k (j-1)} \cdot \varepsilon_j \Theta_{L / k , S} (j)^\#, 
\]
where $\Theta_{L / k , S} (j)$ is the value of the equivariant $L$-function at $s=j$ 
(see \S \ref{ss:L-function}). 

\item[(2)]
For any positive integers $j$ and $j'$, we have 
\[
\Twist_{1-j, 1-j'}^{\Xi} (\varepsilon_j Z_{K_{\infty} /k, S}^{j}) = \varepsilon_{j'} Z_{K_{\infty} /k, S}^{j'},  
\]
where $\Twist_{1-j, 1-j'}^{\Xi} : \varepsilon_j\Xi_{K_\infty / k, S} (j) \xrightarrow{\sim} 
\varepsilon_{j'}\Xi_{K_\infty / k, S} (j') $ is the twist map which is defined in \S \ref{ss:twist_CM}. 
\end{itemize}
\end{thm}
Theorem \ref{thm:zeta>0} (1) asserts that the $\varepsilon_j$-component of the ETNC for $L / k$ and $\Z_p (j)$ is valid for any $L \in \Omega (K_\infty / k)$ and any positive integer $j$. 
 Moreover, (2) asserts that these zeta elements are compatible with Tate twists. 
 In this sense, Theorem \ref{thm:zeta>0} can be regarded as a partial answer to the GIMC.
 
A key point to prove the first main result is a predicted relation between zeta elements like a ``functional equation". 
This predicted relation is called the global $\varepsilon$-conjecture (\cite{Katopre}) or the local equivariant Tamagawa number conjecture. 
In recent work \cite{ADK}, the author, Dainobu, and Kataoka study the local equivariant Tamagawa number conjecture for Tate motives and 
prove this under the same unramified condition in Theorem \ref{thm:zeta>0}. 
Moreover, from the minus component of the ETNC proved by Dasgupta--Kakde--Silliman \cite{DKS}, 
we obtain a zeta element that interpolates the values of $L$-functions at all non-positive integers (Theorem \ref{thm:zeta<1} below). 
Combining this with the result of \cite{ADK}, we prove Theorem \ref{thm:zeta>0} in \S \ref{sec:func_eq}. 

\begin{rem}
The unramified assumption in Theorem \ref{thm:zeta>0} is needed only to ensure the validity of the main result of \cite{ADK} (Theorem \ref{thm:local_ETNC} below).
If Theorem \ref{thm:local_ETNC} can be established without the unramified assumption, 
then Theorem \ref{thm:zeta>0} can also be proved without this assumption. 
\end{rem}

Next we introduce the second main result of this paper. 
We assume that the conditions in Theorem \ref{thm:zeta>0} are satisfied, namely, $k / \Q$ is unramified at $p$ and $K/k$ is also unramified at all $p$-adic primes. 
From the zeta element constructed in Theorem \ref{thm:zeta>0}, 
we can define an element 
\[
\eta_{L, S}^j \in 
\varepsilon_j \bigcap_{\Z_p [\GG_L]}^{r_k} H^1 (G_{L, S} , \Z_p (1-j))
\]
for any positive integer $j$ and any $L \in \Omega(K_\infty / k)$ (Definition \ref{def:eta} below). 
Here $\bigcap^{r_k} (-)$ is the $r_k$-th exterior power bidual defined in \S \ref{ss:def_eta}. 
We briefly describe some properties of this element below. 
A precise statement is given in \S \ref{ss:main2_proof}.

\begin{thm}\label{thm:ES}
For any positive integer $j$ and any $L \in \Omega(K_\infty / k)$, 
%the element $\eta_{L, S}^j \in 
%\varepsilon_j \bigcap_{\Z_p [\GG_L]}^{r_k} H^1 (G_{L, S} , \Z_p (1-j))$
the following statements hold. 
\begin{enumerate}
\item[(i)]
The element $\eta_{L, S}^j $ determines the initial Fitting ideal of $\varepsilon_j H^2 (G_{K, S} , \Z_p (1-j))$ (Theorem \ref{thm:Euler sys and Fitt} (i)). 
\item[(ii)] 
The image of $\eta_{L, S}^j$ under the period regulator isomorphism coincides with the values of $L$-functions at $s=j$ (Theorem \ref{thm:Euler sys and Fitt} (ii)). 
\item[(iii)]
For any positive integers $j$ and $j'$, 
the elements $\eta_{L, S}^j$ and $\eta_{L, S}^{j'}$ satisfy a certain congruence relation (Theorem \ref{thm:generalised kummer congruence}). 
\end{enumerate}
\end{thm}

Let us discuss a conjecture of Burns--Kurihara--Sano \cite{BKS20} and its relation to Theorem \ref{thm:ES}. 
For a finite abelian extension of number fields $K/k$ and an integer $j$, 
Burns, Kurihara, and Sano predict the existence of an element in an exterior power bidual of $H^1 (G_{K, S}, \Z_p (1-j))$ which is related to the leading term of the associated $L$-function at $s=j$ (\cite[Definition 2.9 and Conjecture 3.6]{BKS20}). 
When $j = 0$, this predicted element coincides with the Rubin--Stark element; for this reason, 
they called it a {\it generalized Stark element}. 
They further conjecture that generalized Stark elements satisfy a certain congruence relation as $j$ varies (cf. \cite[Conjecture 2.11]{Tsoi19}). Theorem \ref{thm:ES} shows that these conjectures hold in our setting. 
More precisely, the element $\eta_{L, S}^j$ coincides with the generalized Stark element, and 
the congruence relation in Theorem \ref{thm:ES} (iii) agrees with that predicted by Burns--Kurihara--Sano. 

In \cite{Kato91}, Kato proved that the image of the cyclotomic unit (or its twist) under the dual exponential map coincides with  the value of the $L$-function. 
When the base field $k$ is a rational field $\Q$, 
it follows from this result that the element $\eta_{L, S}^j$ coincides with the cyclotomic unit. 
We discuss this in \S \ref{k=Q}.

\begin{rem}\label{rem:ES}
(1) Let $\cK$ be the maximal abelian CM-extension of $K$ which is unramified at all $p$-adic primes. 
We set  $\cK_\infty : = \cup_{n \geq 0}\cK(\mu_{p^n})$ and  
\[
\Omega (\cK_\infty / k) := \{ k \subset L \subset \cK_\infty \; | \; [L : k] < \infty \text{ and } L \text{ is a CM-field} \}. 
\]
Let $j$ be a positive integer. 
Then, for $L \in \Omega (\cK_\infty / k) $, 
we can define the element $\eta_{L, S}^j  \in \varepsilon_j \bigcap_{\Z_p [\GG_L]}^{r_k} H^1 (G_{L, S} , \Z_p (1-j))$ 
which satisfies the statements of Theorem \ref{thm:ES}. 
Then, the system $\{ \eta_{L, S}^j \}_{L \in\Omega (\cK_\infty / k) }$ satisfies a certain norm relation (this can be checked as in \cite[Proposition 2.12]{BKS20}). 
More precisely, 
this system is an Euler system of rank $r_k$ for $(\Z_p(j), \cK_\infty)$ in the sense of \cite[Definition 2.3]{BS19}. 

(2) The general theory of higher rank Euler systems, established by Burns--Sano \cite{BS19} and Burns--Sano--Sakamoto \cite{BSS}, 
clarifies the relationship between higher rank Euler systems and the {\it higher} Fitting ideals of arithmetic modules.
However, in our case, the technical conditions \cite[Hypothesis 3.2 and 3.3]{BSS} are not satisfied in general: 
the mod $p$ reduction of the representation must not contain the trivial character or the Teichm\"{u}ller character. 
This is why we study only the initial Fitting ideal in Theorem \ref{thm:ES} (i). 
By taking non-trivial character components, we may apply their result to our Euler system, 
but we do not discuss this in this paper.
\end{rem}

%%%%%%%%%%%%%%%%%%%%%%%%%%%%%%%%%%%%%
%\subsection{Organization of this paper}\label{ss:organization}
%%%%%%%%%%%%%%%%%%%%%%%%%%%%%%%%%%%%%%%

%In \S \ref{Preliminary}, we review the theory of determinant modules and give some algebraic notation.   

%In \S \ref{sec:real}, we review local and global Galois cohomology complexes. 

%In \S \ref{s:zeta<1}, we discuss Iwasawa theory for $\mathbb{G}_m$ over a totally real field.

%In \S \ref{sec:func_eq}, we prove Theorem \ref{thm:zeta>0} combining the result of \S \ref{s:zeta<1} with a main result of \cite{ADK}. 

%In \S \ref{section:eulersystem}, we give an explicit statement of Theorem \ref{thm:ES} and prove this. 

%In \S \ref{k=Q}, we prove that the element $\eta_{L, S}$ coincides with the cyclotomic unit when the base field $k=\Q$.  

%%%%%%%%%%%%%%%%%%%%%%%%%%%
\subsection{Acknowledgments}\label{s:ack}
%%%%%%%%%%%%%%%%%%%%%%%%%%%
The author is very grateful to Naoto Dainobu and Takenori Kataoka for many interesting discussions, encouragement, and advice. 
He would also like to thank Takamichi Sano and Masato Kurihara for their helpful advice and encouragement. 
This work was supported by JSPS KAKENHI Grant Numbers 23KJ1943.

%%%%%%%%%%%%%%%%%%%%%%%%%%%%%%%%%%%%%%%%%%%%%%%%%%%%%%%%%
\section{Preliminary}\label{Preliminary} 
\subsection{Determinant module}\label{App:det}
We review the theory of the determinant module (cf.~\cite{KM76}). 
Throughout this paper, we use a notation for the determinant module in this subsection. 

Let $R$ be a commutative ring. 
We consider a category $\mathcal{P}_R$ of graded invertible $R$-modules whose objects
are pairs $(L, r)$ where $L$ is an invertible $R$-module and $r : \Spec(R) \to \Z$ is a
locally constant function, whose morphisms between $(L,r)$ and $(M,s)$ are 
isomorphisms $L \xrightarrow{\sim} M$ as $R$-module if $r = s$, or empty otherwise. 

The category $\mathcal{P}_R$ is equipped with the structure
of a (tensor) product defined by $( L , r ) \otimes_{R} (M,s) := (L \otimes_R M, r+ s)$ with the
natural associativity constraint and the commutativity constraint 
$\psi: (L,r) \otimes_{R} (M,s) \xrightarrow{\sim} (M,s) \otimes_{R} (L,r) : l \otimes m 
\mapsto (-1)^{rs} m \otimes l$. 
From this, we always identify 
$ (L,r) \otimes_{R} (M,s) = (M,s) \otimes_{R} (L,r) $. 
The unit object for the product is ${\trivial}_R := (R,0)$. 

 For each $(L,r) \in \mathcal{P}_R$, define 
 $ (L,r)^{-1} :=(\Hom_R(L,R),-r)$. 
 This becomes an inverse of $(L,r)$ by the evaluation map  
 $\ev_{(L, r)} : (L,r) \otimes_R (L,r)^{-1} \to {\trivial}_R$ 
 induced by $ L \otimes_R \Hom_R (L,R) \to R : x \otimes f \mapsto f(x)$.  
Similarly, we also have the evaluation map 
$\ev_{(L, r)^{-1}} : (L,r)^{-1} \otimes_R (L,r) \to {\trivial}_R$
induced by $  \Hom_R (L,R)  \otimes_R  L  \to R: f \otimes x \mapsto f(x)$. 
Here we note that $\ev_{(L, r)^{-1}} = (-1)^r \cdot \ev_{(L,r)}$, 
where we identify $(L,r) \otimes_R (L,r)^{-1} \overset{\psi}{\simeq} (L,r)^{-1} \otimes_R (L,r)$ as above. 

Throughout this paper, we write the evaluation map ``$\ev$" for simplicity. 
By the canonical isomorphism $L \simeq L^{\ast \ast} ; x \mapsto (f \mapsto f(x))$, 
we always identify $(L, r) = ((L, r)^{-1})^{-1}$. 

 For a ring homomorphism $f : R \to R'$, 
one has a base change functor $R' \otimes_R (-) : \mathcal{P}_{R} \to \mathcal{P}_{R'}$ defined by
$ (L,r) \mapsto (R' \otimes_R L, r \circ f' )$ where $f' : \Spec(R') \to \Spec(R)$ is induced by $f$.

For a finitely generated projective $R$-module $P$, its determinant module is defined as the exterior power 
\[
\Det_R(P) := \Big{(} \bigwedge^{\rank_R(P)} P ,\rk_R(P) \Big{)} \in \mathcal{P}_R
\]
where $\rank_R(P)$ denotes the (locally constant) rank of $P$ as an $R$-module. 
We write $\Det_R^{-1}(P)$ for the inverse of $\Det_R(P)$ in $\mathcal{P}_R$.  

Let ${\bf P}_{\rm is} (R)$ denote the category whose objects are finitely generated projective $R$-modules and 
whose morphisms are isomorphisms of $R$-modules. 
For any $P_1 , P_2 \in {\bf P}_{\rm is} (R)$ with rank $r$ and an isomorphism $f : P_1 \xrightarrow{\sim} P_2$, 
we have a canonical isomorphism $\Det_R(P_1) \xrightarrow{\sim} \Det_R(P_2) ; \wedge_{i=1}^{r} x_i \mapsto 
\wedge_{i=1}^{r} f(x_i)$. 
From this, we obtain a functor $ {\bf P}_{\rm is} (R) \to  \mathcal{P}_R ; P \mapsto \Det_R(P)$. 

For any exact sequence of $R$-modules $0 \to P_1 \xrightarrow{f} P_2 \xrightarrow{g} P_3 \to 0$ such that 
$P_1, P_2, P_3 \in {\bf P}_{\rm is} (R)$, 
we have the following canonical isomorphism between determinant modules; 
\begin{equation}\label{det_exact}
 \Det_R(P_1) \otimes_R \Det_R(P_3)  \simeq \Det_R(P_2) 
 \end{equation}
 induced by 
 \begin{align*}
 x_1 \wedge \cdots \wedge x_{r_1} \otimes 
 y_{1} \wedge \cdots \wedge y_{r_3}  
\mapsto 
f(x_1) \wedge \cdots \wedge f(x_{r_1}) \wedge  
 \widetilde{y_{1}} \wedge \cdots \wedge \widetilde{y_{r_3}} . 
\end{align*}
Here, we put $r_i := \rk_R (P_i)$ ($i= 1, 3$), $ x_1 , \dots , x_{r_1} $ 
(resp. $ y_{1} , \dots , y_{ r_3}$) are local section of $P_1$, (resp. $P_3$) and $\widetilde{y_i} \in P_2$ ($1 \leq i \leq r_3 $) is a lift of $y_i$ by the map $g$. 
We also have a canonical isomorphism: 
\[
 \Det_R^{-1} (P_3) \otimes_R \Det_R^{-1} (P_1)  \simeq \Det_R^{-1} (P_2). 
\]
such that the following diagram is commutative: 
\[
\xymatrix{
 \Det_R(P_1) \otimes_R \Det_R(P_3)  \otimes_R
 \Det_R^{-1}(P_3) \otimes_R \Det^{-1}_R(P_1)  
 \ar[r]^-{\sim} \ar[d]^{\sim}_-{\ev_{\Det(P_3)}}
 &  \Det_R(P_2) \otimes_R  \Det_R^{-1}(P_2) 
 \ar[d]^-{\sim}_-{\ev_{\Det(P_2)}}
\\
\Det_R(P_1) \otimes_R  \Det_R^{-1}(P_1) \ar[r]^-{\sim}_-{\ev_{\Det(P_1)}}
&
R, 
}
\]
 where the upper horizontal arrow is induced by the above isomorphisms and 
 the both vertical arrows and the bottom horizontal arrow are induced by the evaluation map 
 of $ \Det_R(P_3),  \Det_R(P_2) $ and  $\Det_R(P_1)$ respectively.  

We state an elementary lemma. 
\begin{lem}(\cite[pages 23--24]{KM76})\label{lem:nine_term}
Given a commutative diagram of exact sequences: 
\[
\xymatrix{
A_1 
\ar@{^(->}[r] \ar@{_(->}[d]
&A_2 
\ar@{->>}[r] \ar@{_(->}[d]
&A_3 
\ar@{_(->}[d]
\\
B_1 
\ar@{^(->}[r] \ar@{->>}[d]
&B_2 
\ar@{->>}[r] \ar@{->>}[d]
&B_3 
\ar@{->>}[d]
\\
C_1 
\ar@{^(->}[r]
&C_2 
\ar@{->>}[r]
&C_3, 
}
\]
where all modules in the diagram are in ${\bf P}_{\rm is} (R)$. 
Then, the following diagram is commutative: 
\[
\xymatrix{
\Det_R (B_2)
\ar[r]^-{\sim}
\ar[dd]^-{\sim}
&
\Det_R (A_2) \otimes_R \Det_R (C_2)
\ar[d]^-{\sim}
\\
&
\Det_R (A_1) \otimes_R \Det_R (A_3)
\otimes_R \Det_R (C_1) \otimes_R \Det_R (C_3)
\ar[d]^-{\sim}_-{\psi}
\\
\Det_R (B_1) \otimes_R \Det_R (B_3)
\ar[r]^-{\sim}
&
\Det_R (A_1) \otimes_R \Det_R (C_1)
\otimes_R \Det_R (A_3) \otimes_R \Det_R (C_3), 
}
\]
where all isomorphisms except $\psi$ are \eqref{det_exact} that is induced by the exact sequence in the diagram.  
\end{lem}

For any perfect complex $C= [ \cdots C^i \to C^{i+1} \to \cdots]$ of $R$-modules, 
 its determinant module is defined by 
 \[
 \Det_R(C) := \bigotimes_{i \in \Z} \Det_R^{(-1)^i} (C^i)  \in \mathcal{P}_R. 
 \]
Let $D^{\perf} (R)$ be the derived category of perfect complexes of $R$-modules. 
Then, this determinant can be extended to a functor 
\[ 
D^{\perf} (R) \to \mathcal{P}_R ;  C \mapsto \Det_R(C)
\]
which is called the determinant functor. 
Then, for any quasi-isomorphism $f : C_1 \xrightarrow{\sim} C_2$ with $C_1 , C_2 \in D^{\perf}(R)$, 
we have two canonical isomorphisms $f_1: \Det_R (C_1) \xrightarrow{\sim} \Det_R (C_2)$ and 
$f_2: \Det_R^{-1} (C_1) \xrightarrow{\sim} \Det_R^{-1} (C_2)$ such that 
the following diagram is commutative: 
\[
\xymatrix{
\Det_R(C_1) \otimes_R \Det_R^{-1} (C_1) 
\ar[r]^-{\ev}_-{\sim} 
\ar[d]^-{\sim}_-{f_1 \otimes f_2}
&{\trivial}_R
\ar@{=}[d]
\\
\Det_R(C_2) \otimes_R \Det_R^{-1} (C_2) 
\ar[r]^-{\ev}_-{\sim} 
&{\trivial}_R
}
\]
We denote  both isomorphisms $f_1$ and $f_2$ by $f$ for simplicity, and this will cause no confusion. 

Finally, we enumerate some properties of the determinant module. 
\begin{itemize}
\item[(1)] 
For any exact triangle $C_1 \to C_2 \to C_3$ in $D^{\perf} (R)$, 
there is a canonical isomorphism 
$
 \Det_R(C_1) \otimes_R \Det_R(C_3)  \simeq \Det_R(C_2).  
$
\item[(2)] For any ring homomorphism $f : R \to R'$ and any $P \in {\bf P}_{\rm is} (R)$,  
(resp. $C \in D^{\perf} (R)$), 
we have a canonical isomorphism 
\[
R' \otimes_R \Det_{R} (P) \simeq \Det_{R'} (R' \otimes_R P) 
\;\;\;
(\text{resp. }  
R' \otimes_R \Det_{R} (C) \simeq \Det_{R'} (R' \otimesL_R C) 
).
\]
\item[(3)]
For any $C \in D^{\perf} (R)$ such that $H^i (C)  \in {\bf P}_{\rm is} (R)$ for each $i$,  
there is a canonical isomorphism 
\[
\Det_{R} (C) \simeq \bigotimes_{i  \in \Z} \Det_R^{(-1)^i} (H^i (C)). 
\]
In particular, if $C$ is acyclic, we have a canonical isomorphism 
$\Det_{R} (C) \simeq {\trivial}_R$. 
\end{itemize}

%For any $R$-module $M$, put $M^\ast := \Hom_R (M, R)$ and for $(L, r) \in \mathcal{P}_R$, 
%define $(L. r)^\vee := (L^\ast , r)$. 
%Note that this induces an anti equivalence 
%$(-)^\vee : \mathcal{P}_R \xrightarrow{\sim} \mathcal{P}_R$ 
%and clearly, we know $(L , 0)^\vee = (L, 0)^{-1}$. 
%For any morphism $f :[L, r] \xrightarrow{\sim} [M,r]$ in $\mathcal{P}_R$, 
%we write $f^{\vee} : [M, r]^\vee \xrightarrow{\sim} [L,r]^\vee $ for the morphism in $\mathcal{P}_R$ 
%which is correspodense to the morphism $f$ by the anti equivalence $()^\vee$. 

%Then, for any $P \in {\bf P}_{\rm is} (R)$, 
%we have an isomorphism
%\[
%\Det_R (P^\ast) \xrightarrow{\sim} \Det_R (P)^\ast
%\]
%which is defined by 
%\[
%f_1 \wedge \cdots \wedge f_r \mapsto \parenth{
%x_1 \wedge \cdots \wedge x_r \mapsto \sum_{\sigma \in \mathfrak{S}_r} \sgn(\sigma) f_1 (x_{\sigma(1)}) \cdots f_r (x_{\sigma(r)}) 
%}.
%\] 
%Similarly, for any $C \in D^{\perf} (R)$, we also have an isomorphism
%\[
%\Det_R (\RHom_{R} (C, R)) \simeq \Det (C)^\vee. 
%\]

The determinant modules should be treated as graded invertible modules like this, 
but to ease the notation we omit explicit references to the grades. 
In particular, we denote ${\trivial}_R $ by $R$ for simplicity. 

%%%%%%%%%%%%%%%%%%%%%
\subsection{Algebraic notation}\label{ss:alg_notation}
%%%%%%%%%%%%%%%%%%%%%
Throughout this paper, we use the following notation. 

Let $R$ be a commutative ring and $G$ a finite abelian group. 
Then, we write
\[
\# : R[G] \to R[G], 
\quad
x \mapsto x^\#
\]
for the $R$-algebra automorphism of $R[G]$ (the involution)
defined by $\sigma \mapsto \sigma^{-1}$ for any $\sigma \in G$. 
%For any $x \in R[G]$, we write $x^\#$ for the image of $x$ under the involution $\#$.  
For any $R[G]$-module $M$, 
let $M^\#$ be the $R[G]$-module that is the same as $M$ as an $R$-module but $G$ acts on $M^\#$ via the involution $\#$. 
%Similar notations will be used for $R[[\GG]]$-modules when $\GG$ is a profinite group. 

For any $R[G]$-module $M$, put $M^\ast := \Hom_R(M, R)$. 
We define an action of $G$ on $M^\ast$ by 
\[
(\sigma \cdot f) (a) := f (\sigma^{-1} a) ,\;\;\;\; (\sigma \in G, f \in M^\ast, a \in M). 
\]
Then, we have an $R[G]$-isomorphism
\begin{equation}\label{homRRG}
\Hom_R (M, R) 
\stackrel{\sim}{\rightarrow} 
\Hom_{R[G]} (M, R[G])^\# \; ;
f \mapsto 
\sum_{\sigma \in G} f(\sigma(\cdot))\sigma^{-1}. 
\end{equation}
For any $R[G]$-homomorphism $\phi: M \to N$, 
we write $\phi^\ast : N^\ast \to M^\ast$ for the $R[G]$-homomorphism which is induced by $\phi$.

%From this, for any $P \in {\bf P}_{\rm is} (R[G])$ (resp.~$C \in D^{\perf} (R[G])$),  
%we have a canonical isomorphism 
%\begin{equation}\label{Det_Hom}
%\Det_{R[G]} (P) \simeq \Det_{R[G]} (\Hom_{R} (P, R))^{\#, \vee} , 
%\parenth{\text{resp.}  \Det_{R[G]} (C) \simeq \Det_{R[G]} (\RHom_{R} (C, R))^{\#, \vee}}, 
%\end{equation}
%where 
%$
%(-)^\vee : \mathcal{P}_{R[G]} \xrightarrow{\sim} \mathcal{P}_{R[G]}
%$ 
%is the anti equivalence in \S \ref{App:det}. 

%\[
%M^* := \Hom_R (M, R) = \Hom_{R[G]} (M, R[G])^\#. 
%\] 
%Then the functor $\Det_{R[G]}^{-1} (-)$ is equal to $\Det_{R[G]}(-)^{*, \#}$. 

%%%%%%%%%%%%%%%%%%%%%%%%%%%%%%%%%%%%%%%%%%%%%%%%%%%%%%%%%
\section{Galois cohomology complex}\label{sec:real}
%%%%%%%%%%%%%%%%%%%%%%%%%%%%%%%%%%%%%%%%%%%%%%%%%%%%%%%%%

In this section, 
we review the local and global Galois cohomology complexes. 
We fix an odd prime $p$. 
We also fix an isomorphism $\C \simeq \C_p$ and always identify them. 

\subsection{Notation}\label{ss:notation}
Let $K/k $ be a finite abelian CM-extension. 
We put $G := \Gal (K/k)$.  
We write $\Sigma_K = \{ \iota : K \hookrightarrow \C \}$ for the set of embeddings, so $\# \Sigma_K  =[K: \Q]$. 
We put 
%$r_k := [k : \Q]$ and 
\[
X_K (j) = \bigoplus_{\iota \in \Sigma_K} \Z_p (j). 
\]
As in \cite[\S 2.1]{BKS20}, we fix a $\Z_p$-basis of $X_K (j)$ as follows. 
We set $e_{p^\infty} := (\exp(\frac{2 \pi  \sqrt{-1}}{p^n}))_n  \in \Z_p (1)$. 
Then, we obtain a $\Z_p$-basis $\{ e_\iota^j \mid \iota \in \Sigma_K\}$ of $X_K (j)$, 
where $e_{\iota}^j$ is the element of $X_K (j)$ whose 
$\iota$-component is $e_{p^\infty}^{\otimes j}$ and other components are zero. 

We also consider the Betti cohomology space
\[
H_K (j) := \bigoplus_{\iota \in \Sigma_K} (2\pi \sqrt{-1})^j \Q
\]
for any $j \in \Z$. 
Let $b_{\iota}^j$ be the element of $H_K (j)$ whose 
$\iota$-component is $(2 \pi \sqrt{-1})^j$ and other components are zero. 

The group $G$ acts on $\Sigma_K$ by $\sigma \cdot \iota := \iota \circ \sigma^{-1}$ ($\sigma \in G, \iota \in \Sigma_K$).
Then, $G$ also acts on $H_K (j)$ and $X_K(j)$ by $\sigma \cdot (a_\iota)_\iota := (a_\iota)_{\iota \circ \sigma^{-1}}$, 
making them free modules of rank $[k:\Q]$ as a $\Q [G]$-module and a $\Z_p [G]$-module, respectively.
In addition, the complex conjugation $c_{\R} \in \Gal (\C / \R)$ acts on $H_K (j)$ and $X_K (j)$ by 
%$c_{\R} \cdot(a_\iota)_\iota := (c_{\R} \cdot a_\iota)_{c_{\R} \circ \iota}$ 
$c_{\R} \cdot(a_\iota)_\iota := ((-1)^j \cdot a_\iota)_{c_{\R} \circ \iota}$. 
We write $H_K (j)^+$ (resp. $X_K (j)^+$)
for the subspace of $H_K (j)$ (resp. $X_K (j)$) on which the complex conjugation $c_{\R}$ acts trivially. 
Then, we obtain a canonical $\Q_p [G \times \Gal (\C / \R)]$-isomorphism 
\begin{equation}\label{Y Betti}
\Q_p \otimes_{\Z_p} X_K(j)
% \simeq \bigoplus_{\iota \in \Sigma_K} \Q_p (j)
\stackrel{\sim}{\rightarrow}  \Q_p \otimes_\Q H_K (j)
\end{equation}
by sending $1 \otimes e_{\iota}^j$ to $1 \otimes b_{\iota}^j$. 

Next, we consider an $\R[G]$-isomorphism
\begin{equation}
\R \otimes_\Q K \simeq
\R \otimes_\Q (H_K (j)^+ \oplus H_K (j-1)^+ )
\end{equation}
that is given by 
\[
1 \otimes x \mapsto \begin{cases}
\left(\left(\Real(\iota(x))\right)_{\iota \in \Sigma_K}, \left(\sqrt{-1} \Imag(\iota(x)) \right)_{\iota \in \Sigma_K} \right)
& \text{if } j \text{ is even,} \\
\left(\left(\sqrt{-1} \Imag(\iota(x)) \right)_{\iota \in \Sigma_K}, \left( \Real(\iota(x))\right)_{\iota \in \Sigma_K} \right)
& \text{if } j \text{ is odd,} 
\end{cases}
\]
for $x \in K$. 
Combining this with \eqref{Y Betti}, we obtain a $\C_p[G]$-isomorphism 
\begin{equation}\label{embedding}
\gamma_{K}^j:\C_p \otimes_\Q K \simeq
\C_p \otimes_{\Z_p} (X_K (j)^+ \oplus X_K (j-1)^+ ). 
\end{equation}

On the other hand, we consider a $\Z_p[G]$-isomorphism 
\begin{equation}\label{2zure}
\beta_{K}^j: X_K (j)^+ \oplus X_K (j-1)^+
 \stackrel{\sim}{\rightarrow} X_K (j)
\end{equation}
 that is given by 
 \[
 ((x_\iota)_{\iota \in \Sigma_K},(y_\iota)_{\iota \in \Sigma_K}) \mapsto 
(2 x_\iota +  \frac{1}{2} (y_\iota \otimes e_\iota^1))_{\iota \in \Sigma_K }.  
\]
By composing \eqref{embedding} with this, 
we obtain a $\C_p[G]$-isomorphism
\[
\alpha^j_{K}:= \beta_{K}^j \circ \gamma_K^j : \C_p \otimes_\Q K \xrightarrow{\sim}
 \C_p \otimes_{\Z_p} X_K (j)  
\]
for any $j \in \Z$.

%%%%%%%%%%%%%%%%%%%%%%%%%%%%%%%%%%%%%%%%%%%%%%%%%%%%
\subsection{Local Galois cohomology complex}\label{ss:local}
%%%%%%%%%%%%%%%%%%%%%%%%%%%%%%%%%%%%%%%%%%%%%%%%%%%%

For any prime $v$ of $k$, we consider a semi-local Galois cohomology complex which is given by
\[
\RG(K_v, \Z_p (j)) := \bigoplus_{w|v} \RG(K_w, \Z_p (j)) 
\]
for any integer $j$, where $w$ runs over all primes of $K$ above $v$. 

Let $S$ be a finite set of places of $k$ which contains $S_\infty \cup S_p \cup S_{\ram} (K/k)$ as in the Introduction. 
We write $S_f$ for the subset of finite primes in $S$.

\begin{defn}\label{def:Xi_loc}
We define a graded invertible $\Z_p [G]$-module $\Xi^{\loc}_{K/k, S} (j)$ by 
\[
\Xi_{K / k, S}^{\loc} (j)
= \parenth{\bigotimes_{v \in S_f}  \Det_{\Z_p[G]}^{-1}  \Big{(}\RG (K_v , \Z_p (j)) \Big{)}} 
\otimes_{\Z_p[G]}
\Det_{\Z_p[G]}^{-1} (X_K (j)).
\]
We note that the grade of this is zero. 
\end{defn}

The goal of this subsection is to define an isomorphism 
\[
 \Phi_{K/k, S}^{\loc, j}  : \C_p \otimes_{\Z_p} \Xi_{K / k, S}^{\loc} (j) \xrightarrow{\sim} \C_p [G]
\]
for any $j \in \Z_{\leq 0}$ as in \cite[\S 8.3]{ADK}. 

For a $\Q_p[G]$-module $M$, we set $M^\ast := \Hom_{\Q_p} (M, \Q_p) \overset{\eqref{homRRG}}{\simeq} \Hom_{\Q_p[G]} (M, \Q_p[G])^\#$. 

For any finite prime $v$ of $k$ and $j \in \Z_{\leq 0}$, 
we consider a perfect pairing induced by the cup product and the invariant maps
 (i.e., the local duality)
\[
H^1 (K_v, \Q_p (j)) \times H^1 (K_v, \Q_p (1-j)) \xrightarrow{\cup} H^2 (K_v, \Q_p (1)) \xrightarrow{\sum_{w\mid v} \inv_{w}} \Q_p. 
\]
Using this, for any $p$-adic prime $v$ of $k$ and $j \in \Z_{\leq 0}$,
we consider the Bloch-Kato dual exponential map 
\[
\exp_{\Q_p (j)}^\ast : H^1 (K_v, \Q_p (j)) \xrightarrow{\sim}
H^1 (K_v, \Q_p (1-j))^\ast 
\xrightarrow{(\exp_{\Q_p (1-j)})^\ast} 
K_v^\ast,
\]
where we put $K_v := k_v \otimes_{k} K \simeq \oplus_{w \mid v} K_w$, 
the first isomorphism is induced by the above perfect pairing, 
and the second arrow is the dual of the exponential map $\exp_{\Q_p (1-j)}$. 
Here we note that some authors use different conventions concerning the order of cup products. 

Next we define an isomorphism
\[
\begin{cases}
\vartheta_{K_v}^{j} : 
 \Det_{\Q_p[G]}^{-1} (\RG(K_v, \Q_p (j)))
\xrightarrow{\sim} 
\Q_p [G]
&\text{if } v \nmid p, \\
\phi_{K_v}^j:  \Det^{-1}_{\Q_p[G]} (\RG(K_v, \Q_p (j)))
\xrightarrow{\sim} 
\Det_{\Q_p[G]}(K_v^\ast)
&\text{if }v \mid p, 
\end{cases}
\]
for a finite prime $v$ of $k$ and $j \in \Z_{\leq 0}$. 
The construction depends on whether $j = 0$ or $j < 0$.

Suppose $j<0$ and $v \nmid p$.
Then, the complex 
$\RG(K_v, \Q_p (j))$ is acyclic and the isomorphism $\vartheta_{K_v}^{j}$ is defined by the canonical one. 

Suppose $j<0$ and $v \mid p$.
Then, the complex $\RG(K_v, \Q_p (j))$ is acyclic outside degree one and the dual exponential map $\exp_{\Q_p (j)}^\ast$ is an isomorphism. 
%Then the canonical map in Proposition \ref{prop:det_ses4} and $\exp_{\Q_p (j)}^\ast$
These facts induce the isomorphism $\phi_{K_v}^j$.

Suppose $j = 0$. For each finite prime $v$ of $k$, 
we consider the following composite map 
\begin{equation}\label{rec_v}
 H^0 (K_v, \Q_p)
 =
\bigoplus_{w \mid v} \Q_p  
\xrightarrow{(\ord_{K_v})^\ast}
(\Q_p \otimes_{\Z_p} \bigoplus_{w \mid v}  \widehat{K_w^\times})^\ast 
\simeq 
\Hom ( \bigoplus_{w \mid v} G_{K_w}^{\ab} , \Q_p)
=
H^1 (K_v, \Q_p),  
\end{equation}
where $\widehat{K_w^\times}$ is the $p$-adic completion of $K_w^\times$, 
$\ord_{K_v} : \oplus_{w \mid v}  \widehat{K_w^\times} \to \oplus_{w \mid v} \Z_p$ is the normalized valuation map, 
 $G_{K_w}^{\ab}$ is the Galois group of the maximal abelian extension over $K_w$, 
 and the isomorphism is induced by the reciprocity map of local class field theory.  

Now suppose $j =0$ and $v \nmid p$. 
In this case, the complex 
$\RG(K_v, \Q_p (0))$ is acyclic outside degree zero and one
and the map \eqref{rec_v} is an isomorphism. 
Then, we define $\vartheta_{K_v}^{0}$ as
\begin{align*}
\Det_{\Q_p[G]}^{-1}(\RG(K_v, \Q_p))
& \simeq  \Det^{-1}_{\Q_p[G]}(H^0(K_v, \Q_p)) \otimes_{\Q_p[G]} \Det_{\Q_p[G]}(H^1(K_v, \Q_p))\\
&\overset{(\ref{rec_v})}{\simeq} \Det^{-1}_{\Q_p[G]}(H^1(K_v, \Q_p)) \otimes_{\Q_p[G]} \Det_{\Q_p[G]}(H^1(K_v, \Q_p))\\
&\overset{\ev}{\simeq} \Q_p[G].
\end{align*}

Finally, we suppose $j = 0$ and $v \mid p$.
Then, the complex $\RG(K_v, \Q_p (0))$ is acyclic outside degree zero and one again. 
In this case, we consider the following commutative diagram: 
\[
\xymatrix{
0 \ar[r]
&
\bigoplus_{w \mid v} \Q_p 
\ar[r]^-{(\ord_{K_v})^\ast} \ar@{=}[d]
&
(\Q_p \otimes_{\Z_p} \bigoplus_{w \mid v} \widehat{K_w^\times} )^\ast
\ar[r]
\ar[d]^-{\sim}
&
(\Q_p \otimes_{\Z_p} \bigoplus_{w \mid v} U_{K_w})^\ast
\ar[r] \ar[d]^-{\sim}_-{(\oplus_{w \mid v}\exp_{p})^\ast}
&0
\\
0 \ar[r]
&
H^0 (K_v, \Q_p)
\ar[r]^-{\eqref{rec_v}}
&
H^1 (K_v, \Q_p)
\ar[r]^-{\exp_{\Q_p}^\ast}
&
K_v^\ast 
\ar[r]
&
0, 
} 
\]
where the middle vertical arrow is induced by the reciprocity map and 
$\exp_p : \Q_p \otimes U_{K_w} \xrightarrow{\sim} K_w$ is the $p$-adic exponential map. 
Since the upper sequence is clearly exact, so is the bottom sequence.
We then obtain a canonical isomorphism
\[
\Det_{\Q_p[G]}(H^1 (K_v, \Q_p)) \simeq \Det_{\Q_p[G]}(H^0 (K_v, \Q_p)) \otimes_{\Q_p[G]} \Det_{\Q_p[G]}(K_v^\ast ).
\]
Using this isomorphism, we define $\phi_{K_v}^0$ as 
\begin{align*}
  \Det_{\Q_p[G]}^{-1}(\RG(K_v, \Q_p))
& \quad  \simeq \Det^{-1}_{\Q_p[G]}(H^0(K_v, \Q_p)) \otimes_{\Q_p[G]} \Det_{\Q_p[G]}(H^1(K_v, \Q_p))\\
& \quad \simeq \Det^{-1}_{\Q_p[G]}(H^0(K_v, \Q_p)) \otimes_{\Q_p[G]} \Det_{\Q_p[G]}(H^0 (K_v, \Q_p)) \otimes_{\Q_p[G]} \Det_{\Q_p[G]}(K_v^\ast )\\
& 
\quad \overset{\ev}{\simeq} 
%\Q_p[G] \otimes_{\Q_p[G]} \Det_{\Q_p[G]}(K_v^\ast )\simeq 
\Det_{\Q_p[G]}(K_v^\ast ).
\end{align*}
This completes the construction of $\vartheta_{K_v}^j$ and $\phi_{K_v}^j$ for $j \leq 0$.

For each $j \in \Z_{\leq 0}$, combining $\vartheta_{K_v}^j$ and $\phi_{K_v}^j$ for each $v \in S_f$ yields an isomorphism 
\begin{equation}\label{dual_version}
\bigotimes_{v \in S_f} \Det_{\Q_p[G]}^{-1} (\RG(K_v, \Q_p (j))) \simeq 
\Det_{\Q_p[G]} ((\Q_p \otimes_{\Q} K )^\ast).
\end{equation}

\begin{defn}\label{def:vartheta_loc}
For any $j \in \Z_{\leq 0}$, we define an isomorphism 
\[
 \Phi_{K/k, S}^{\loc, j}  : \C_p \otimes_{\Z_p} \Xi_{K / k, S}^{\loc} (j) \xrightarrow{\sim} \C_p [G]
\]
by the following composite map  
\begin{align}\label{second_map}
&  \C_p \otimes_{\Z_p} \Xi_{K/k, S}^{\loc} (j) 
= \C_p \otimes_{\Z_p} \left(\bigotimes_{v \in S_f} \Det_{\Z_p[G]}^{-1} (\RG (K_v , \Z_p(j)) )
\otimes \Det_{\Z_p[G]}^{-1} (X_K (j))\right)
 \\& \quad
\overset{\eqref{dual_version} \otimes \eqref{Xpair} }{\xrightarrow{\sim}} 
\C_p \otimes_{\Q_p} \parenth {\Det_{\Q_p[G]} ((\Q_p \otimes K )^\ast) \otimes_{\Q_p[G]} \Det_{\Q_p[G]}^{-1} ((\Q_p \otimes X_K (1-j))^\ast)} 
\\ 
&\quad
\overset{(\alpha_K^{1-j})^{\ast, -1} }{\xrightarrow{\sim}} 
\C_p \otimes_{\Q_p} 
\parenth{
\Det_{\Q_p[G]}( (\Q_p \otimes X_K (1-j) )^\ast) \otimes_{\Q_p[G]} \Det_{\Q_p[G]}^{-1}
( ( \Q_p \otimes X_K (1-j)  )^\ast) 
}
\\
&\quad
\overset{\ev}{\simeq}
\C_p[G], 
\end{align}
where the first isomorphism is induced by \eqref{dual_version} and  
\begin{equation}\label{Xpair}
X_K(j) \simeq \Hom_{\Z_p} (X_K (1-j) , \Z_p)
\end{equation}
that is given by $e_\iota^j \mapsto e_\iota^{1-j, \ast}$. 
Here $e_\iota^{1-j, \ast}$ is the dual basis of $e_\iota^{1-j}$. 

\end{defn}

\begin{rem}\label{rem:sign}
The above isomorphism $ \Phi_{K/k, S}^{\loc, j} $ is equal to the one \cite[equation (8.6)]{ADK}. 
Therefore, we note that $ \Phi_{K/k, S}^{\loc, j} $ coincides with the isomorphism $ \vartheta_{K/k, S}^{\loc, j} $ in \cite[Definition 8.1]{ADK} up to sign $(-1)^{[k:\Q]}$ 
(see \cite[Proposition 8.5]{ADK}). 
\end{rem}

%\MEMO{前の論文と符号をずらしたものに対して同じ記号を使ってしまった。変えるべきか？}

%%%%%%%%%%%%%%%%%%%%%%%%%%%%%%%%%%%%%%%%%%%%%%%%%%%%%%%%
\subsection{Global cohomology complex}\label{ss:period reg map}
%%%%%%%%%%%%%%%%%%%%%%%%%%%%%%%%%%%%%%%%%%%%%%%%%%%%%%%%
In this subsection, we review the global Galois cohomology complex.  
%Recall that $K/k$ be a CM-abelian extension and $G:= \Gal (K/k)$. 
%Let $S$ be a finite set of primes of $k$ such that $S$ contains 
%$S_\infty \cup S_p \cup S_{\ram} (K/k)$ and we set $S_f := S \setminus S_\infty$ as in the previous section. 

For any integer $j$, we define a complex $\Delta_{K, S} (j)$ by  
\[
\Delta_{K, S} (j) := 
{\rm Cone} \Big{(}\RG (G_{K, S} , \Z_p (j)) \to \bigoplus_{v \in S_f} \RG (K_v, \Z_p (j)) \Big{)} [-1],  
\]
where $G_{K, S}$ is the Galois group of the maximal Galois extension over $K$ unramified outside the places 
of $K$ above $S$. 
Then, we have a tautological exact triangle
\begin{equation}\label{triangle}
\Delta_{K, S} (j) \to \RG (G_{K, S} , \Z_p (j)) \to \bigoplus_{v \in S_f} \RG (K_v, \Z_p (j)). 
\end{equation}
For any integer $j$, we set an idempotent $\varepsilon_{j} \in \Z_p[G]$ by 
\[
\varepsilon_j := \frac{1 + (-1)^j c}{2}, 
\]
where $c \in G$ is the complex conjugation. 
%Clearly, we know 
%\[
%\varepsilon_jX_K (j) = X_K (j)^+ \text{ and }  \varepsilon_jX_K (j-1)^+ =0. 
%\]
%Then, there is a tautological exact triangle;
%\begin{equation}\label{tautological triangle}
%\Delta_{K, c} (1-j) \to \RG (G_{K, S(K)} , \Z_p (1-j)) \to \parenth{\bigoplus_{v \in S(K)_f} \RG (K_v , \Z_p (1-j))} 
%\end{equation} 
%For any $\Q_p[G]$-module $M$, 
%we put $M^\ast := \Hom_{\Q_p} (M, \Q_p)$. 

First, we observe the $\varepsilon_{1-j}$-component of Galois cohomology complexes for $\Z_p (j)$ when $j$ is a non-positive integer. 

\begin{lem}\label{lem:complex}
The following statements are valid. 
\begin{itemize} 
\item[(1)] For any $j \in \Z_{<0}$, the complex $\Q_p \otimesL_{\Z_p} \varepsilon_{1-j} \Delta_{K, S} (j) $ is acyclic. 

\item[(2)]  The complex $\Q_p \otimes_{\Z_p} \varepsilon_1 \Delta_{K, S} (0) $ is acyclic outside degree one and two. 
Furthermore, we have a canonical isomorphism 
\[
\Q_p \otimes_{\Z_p} \varepsilon_1 H^i ( \Delta_{K, S} (0)) \simeq 
\begin{cases}
\varepsilon_1 \bigoplus_{w \in S_{f, K}} \Q_p & \text{ if } i=1, \\
\varepsilon_1 (\Q_p \otimes_{\Z} \OO_{K, S}^\times)^\ast & \text{ if } i=2, 
\end{cases}
\]
where $S_{f, K}$ is the set of primes of $K$ above $S_f$ and $\OO_{K, S}^\times$ is the $S$-unit group of $K$. 

\item[(3)] For any $j \in \Z_{\leq 0}$, the complex $\varepsilon_{1-j} \RG (G_{K, S} , \Q_p (j))$ is acyclic outside degree one and 
the following composite map 
\begin{equation}\label{Lemma(3)}
\varepsilon_{1-j} H^1 (G_{K, S} , \Q_p (j)) \to \varepsilon_{1-j}\bigoplus_{v \mid p} H^1 (K_v , \Q_p (j)) \xrightarrow{\exp_{\Q_p(j)}^\ast} 
\varepsilon_{1-j}(\Q_p \otimes_{\Q} K)^\ast
\end{equation}
is an isomorphism, where the first arrow is the natural localization map.  
\end{itemize}
\end{lem}

\begin{proof}
First, we note that by the global duality (cf.~\cite[Proposition 5.4.3]{Nek06}), there is a canonical quasi-isomorphism 
\[
\Delta_{K, S} (j) \simeq \RHom_{\Z_p} (\RG (G_{K, S} , \Z_p (1-j)) , \Z_p)[-3] 
\]
for any integer $j$. Using this, we prove the Lemma. 

(1) Assume that $j$ is a negative integer. 
Then, Soul\'e proved that $H^2 (G_{K, S} , \Q_p (1-j))$ vanishes (cf. \cite[Theorem 10.3.27]{NSW08}). 

On the other hand, by the validity of the Quillen--Lichtenbaum Conjecture which is proved by Rost--Voevodsky and Weibel, 
we have a canonical isomorphism ($p$-adic Chern class map);
\[
H^1 (G_{K, S} , \Z_p (1-j)) \simeq \Z_p \otimes_{\Z} K_{1-2j} (\OO_{K, S}) 
\]
(cf. \cite[Theorem 3.2]{Nic22}), where we write $K_{1-2j}(-)$ for Quillen's higher algebraic $K$-theory functor. 
In addition, 
we have the Borel regulator isomorphism
\[
\R \otimes_\Z K_{1-2j} (\OO_{K, S}) \xrightarrow{\sim} \R \otimes_\Q H_K (-j)^+
\] 
(cf.~\cite[Part I \S 1]{Neu88}) for any negative integer $j$.  
From these, we know that 
\[
\C_p \otimes_{\Q_p} \varepsilon_{1-j} H^1 (G_{K, S} , \Q_p (1-j)) \simeq \C_p \otimes_{\Q} \varepsilon_{1-j}H_K (-j)^+ =0. 
\]
Since the cohomological $p$-dimension of $G_{K, S}$ is two, the claim is valid. 

(2) By Kummer theory and the invariant map of class field theory, 
we have a canonical isomorphism; 
\begin{equation}\label{Q_p(1)}
 H^i (G_{K,S} , \Q_p (1)) \simeq 
\begin{cases}
(\Q_p \otimes_{\Z} \OO_{K, S}^\times) & \text{if } i=1, 
\\
\ker (\bigoplus_{w \in S_{f, K}} \Q_p \xrightarrow{\Sigma} \Q_p) & \text{if } i=2, 
\\ 0 & \text{otherwise.}
\end{cases}
\end{equation}
The claim follows from this and the global duality. 

(3) If $j <0$, by the claim (1) and the exact triangle \eqref{triangle}, 
we have a quasi-isomorphism 
\[
\varepsilon_{1-j} \RG (G_{K, S} , \Q_p (j)) \simeq 
\varepsilon_{1-j} \bigoplus_{v \in S_f} \RG (K_v, \Q_p (j)) 
 \simeq 
 \varepsilon_{1-j} \bigoplus_{v \mid p} \RG (K_v, \Q_p (j)).
\]
In this case, the complex $\bigoplus_{v \mid p} \RG (K_v, \Q_p (j))$ is acyclic outside degree one and 
the dual exponential map is an isomorphism, the claim is valid. 

Next we consider the case when $j = 0$. 
Let $K^+$ be the maximal real subfield of $K$ and we consider the 
restriction map $H^2 (G_{K^+, S} , \Q_p) \to H^2 (G_{K, S} , \Q_p) $. 
Since a composition of the restriction map and the corestriction map is $2$ times, 
this restriction map is injective. 
Furthermore, in this case, 
it is well known that both cohomology groups have the same $\Q_p$-rank (this is known as the Leopoldt defect). 
Therefore, this restriction map is an isomorphism. 
From this, we know $\varepsilon_1 H^2 (G_{K, S} , \Q_p) =0$ and the complex $\varepsilon_1 \RG (G_{K, S}, \Q_p )$ is acyclic outside degree one. 

%Next, we check that the localization map $\varepsilon_{1} H^1(G_{K, S},  \Q_p) \to \varepsilon_{1} \bigoplus_{v \in S_f} H^1 (K_v, \Q_p )$ is injective. 
%Since $\varepsilon_{1} H^0(G_{K, S},  \Q_p) = \varepsilon_{1}  \Q_p = 0$, 
%the map $\varepsilon_{1}\bigoplus_{v \in S_f} H^0 (K_v, \Q_p ) \to \varepsilon_{1}H^1 (\Delta_{K, S} (0))$ induced by \eqref{triangle}
%is injective, thus this is an isomorphism by the claim (2). 
%From this, we obtain the injectivity of the above localization map. 
In addition, since $\varepsilon_{1} H^0(G_{K, S},  \Q_p) = \varepsilon_{1}  \Q_p = 0$, 
the map $\varepsilon_{1}\bigoplus_{v \in S_f} H^0 (K_v, \Q_p ) \to \varepsilon_{1}( \Q_p \otimes_{\Z_p} H^1 (\Delta_{K, S} (0)))$ induced by \eqref{triangle}
is injective, thus this is an isomorphism by the claim (2). 
Furthermore, we know that the localization map $\varepsilon_{1} H^1(G_{K, S},  \Q_p) \to \varepsilon_{1} \bigoplus_{v \in S_f} H^1 (K_v, \Q_p )$ is injective. 

Taking these into account, 
we obtain the commutative diagram: 
\begin{equation}\label{diagram2}
\xymatrix@R=30pt{
&
&
\varepsilon_{1} H^1(G_{K, S},  \Q_p) \ar[r]^-{\eqref{Lemma(3)}} \ar@{^(->}[d]
&
\varepsilon_{1} (\Q_p \otimes_{\Q} K)^\ast \ar@{=}[d]
\\
0 \ar[r]
&
%\displaystyle
\varepsilon_{1} \bigoplus_{v \in S_f} H^0 (K_v, \Q_p )
\ar[r]^-{\eqref{rec_v}} 
\ar[d]^-{\simeq}
&
\varepsilon_{1} \bigoplus_{v \in S_f} H^1 (K_v, \Q_p ) \ar[r]^-{\exp_{\Q_p}^\ast} \ar@{->>}[d]
&
\varepsilon_{1} (\Q_p \otimes_{\Q} K)^\ast \ar[r]
&
0
\\
&
\Q_p \otimes_{\Z_p} \varepsilon_1 H^1 ( \Delta_{K, S} (0)) \ar[d]^-{\simeq}
&
\Q_p \otimes_{\Z_p} \varepsilon_1 H^2 ( \Delta_{K, S} (0)) \ar[d]^-{\simeq}
\\
&
%\displaystyle
\varepsilon_{1} \bigoplus_{w \in S_{f, K}} \Q_p \ar[r]^-{(\oplus_w \ord_{w})^\ast}_-{\simeq}
&
\varepsilon_1 (\Q_p \otimes_{\Z} \OO_{K, S}^\times)^\ast, 
&
}
\end{equation}
%\MEMO{本当は、これが可換になるようにclaim (2)の同型を定めるべし}
where the upper and middle vertical arrows are induced by \eqref{triangle} and 
the bottom vertical isomorphisms are the canonical one in claim (2). 
Moreover, the left bottom horizontal arrow is induced by the normalized valuation map and 
this is an isomorphism by Dirichlet's unit theorem. 
%Clearly, the upper horizontal arrow in the diagram coincides with the composite map in the claim (3) when $j=0$. 
By the snake lemma, we know that \eqref{Lemma(3)} is also an isomorphism when $j=0$.   
\end{proof}

We define a period regulator isomorphism as in \cite{BKS20}. 

\begin{defn}\label{def:period_reg>0}
For any $j \in \Z_{\geq 1}$, we define an isomorphism 
\[
\lambda_{K}^{j} : \C_p \otimes_{\Q_p} \varepsilon_{j}H^1 (G_{K, S} , \Q_p (1-j)) \xrightarrow{\sim} 
\C_p \otimes_{\Z_p} \varepsilon_{j}X_K (-j)^+
\]
by the following composite map 
\begin{align*}
&\C_p \otimes_{\Q_p} \varepsilon_{j} H^1 (G_{K, S} , \Q_p (1-j)) 
%&\to
%\C_p \otimes_{\Q_p} \varepsilon_{1-j} \bigoplus_{v \mid p} H^1 (K_v, \Q_p (j)) 
%\xrightarrow{\exp_{\Q_p (j)}^\ast} 
\overset{\eqref{Lemma(3)}}{\xrightarrow{\sim}}
\C_p \otimes_{\Q_p} \varepsilon_{j} (\Q_p \otimes_{\Q} K)^\ast 
\\
&\overset{(\gamma_K^j)^{\ast, -1}}{\xrightarrow{\sim}}
\C_p \otimes_{\Z_p} \varepsilon_{j} \Hom_{\Z_p} (X_K (j)^+ , \Z_p) 
\overset{\eqref{canonical_id}}{\simeq}
\C_p \otimes_{\Z_p} \varepsilon_{j} X_K (-j)^+, 
\end{align*}
where the second isomorphism is induced by $\gamma_K^j$ defined in \eqref{embedding} 
and the last one is induced by the canonical identification 
\begin{equation}\label{canonical_id}
X_K (-j) \simeq \Hom_{\Z_p} (X_K (j) , \Z_p) .
\end{equation} 
Moreover, we set 
\[
\Xi_{K/k, S} (j) := \Det_{\Z_p[G]}^{-1} (\RG (G_{K, S} , \Z_p (1-j))) \otimes_{\Z_p [G]} 
\Det_{\Z_p[G]}^{-1} (X_K (-j)^+)  
\]
and define an isomorphism 
\[
\vartheta_{K/k, S}^{j} : 
\C_p \otimes_{\Z_p} \varepsilon_{j} \Xi_{K/k, S} (j)
 \xrightarrow{\sim} 
\varepsilon_{j} \C_p [G]
\]
by the following composite map 
\begin{eqnarray*}
&\vartheta_{K/k, S}^j : 
\C_p \otimes_{\Z_p} \varepsilon_{j} \parenth{\Det_{\Z_p[G]}^{-1} (\RG (G_{K, S} , \Z_p (1-j))) \otimes_{\Z_p [G]} 
\Det_{\Z_p[G]}^{-1} (X_K (-j)^+)} 
\\
& \xrightarrow{\sim} 
\C_p \otimes_{\Z_p}\varepsilon_{j} \parenth{\Det_{\Q_p[G]} (H^1 (G_{K, S} , \Q_p (1-j))) \otimes_{\Z_p [G]} 
\Det_{\Z_p[G]}^{-1} (X_K (-j)^+)} 
\\
& \overset{\lambda_{K}^j}{\xrightarrow{\sim}} 
\C_p \otimes_{\Z_p} \varepsilon_{j} \parenth{\Det_{\Z_p[G]} (X_K (-j)^+) \otimes_{\Z_p[G]} \Det_{\Z_p[G]}^{-1} (X_K (-j)^+) }
\overset{\ev}{\xrightarrow{\sim}}
\varepsilon_{j}\C_p [G],   
\end{eqnarray*}
where the first isomorphism is the canonical one induced by the fact that $\varepsilon_{j} \RG (G_{K, S} , \Q_p (1-j))$ is 
acyclic outside degree one (see Lemma \ref{lem:complex} (3)). 
\end{defn}

Similarly, we define an isomorphism $\vartheta_{K/k, S}^{j} $ when $j \leq 0$ as follows. 

\begin{defn}\label{def:period_reg<0}

For any $j \in \Z_{\leq 0}$, we define an isomorphism
\[
\vartheta_{K/k, S}^{j} : 
\C_p \otimes_{\Z_p} \varepsilon_{1-j} \Det_{\Z_p[G]}^{-1} (\Delta_{K,S} (j)) 
 \xrightarrow{\sim} 
\C_p [G]
\]
as follows. 

When $j <0$, the complex $\Q_p \otimesL_{\Z_p} \varepsilon_{1-j}\Delta_{K,S} (j) $ is acyclic by Lemma \ref{lem:complex} (1). 
In this case, the isomorphism $\vartheta_{K/k, S}^{j}$ is defined by the canonical one. 

When $j =0$, we define $\vartheta_{K/k, S}^{0}$ by the following composite map 
\begin{align*}
&\C_p \otimes \varepsilon_1 \Det_{\Z_p[G]}^{-1} (\Delta_{K,S} (0)) 
 \xrightarrow{\sim}
\C_p \otimes_{\Q_p} \varepsilon_1 
\parenth{
\Det_{\Q_p[G]}^{-1} ((\Q_p \otimes_{\Z} \OO_{K, S}^\times)^\ast)}
\otimes
\Det_{\Q_p[G]} (\bigoplus_{w \in S_{f, K}} \Q_p)
\\
& \simeq  
\C_p \otimes_{\Q_p} \varepsilon_1 
\parenth{
\Det_{\Q_p[G]}^{-1} (\bigoplus_{w \in S_{f, K}} \Q_p) \otimes
\Det_{\Q_p[G]} (\bigoplus_{w \in S_{f, K}} \Q_p)}
 \overset{\ev}{\xrightarrow{\sim}} \varepsilon_1\C_p[G],  
 \end{align*}
where the first isomorphism is induced by Lemma \ref{lem:complex} (2) and 
the second one is induced by the following isomorphism 
\[
(\oplus_{w \in S_{f, K}} \ord_{w})^\ast : 
 \varepsilon_1 \bigoplus_{w \in S_{f, K}} \Q_p \xrightarrow{\sim} 
  \varepsilon_1 (\Q_p \otimes_{\Z} \OO_{K, S}^\times)^\ast. 
\]
\end{defn}

%%%%%%%%%%%%%%%%%%%%%%%%%%%%%%%%%%%%%%%%
\subsection{Compatibility of the isomorphisms $\vartheta_{K/k, S}^j$ and $ \Phi_{K/k, S}^{\loc, j} $}
\label{compatibility}
%%%%%%%%%%%%%%%%%%%%%%%%%%%%%%%%%%%%%%%%
In this subsection, we review the compatibility of the isomorphisms 
which are defined in Definition \ref{def:vartheta_loc}, \ref{def:period_reg>0}, and \ref{def:period_reg<0}. 
 
\begin{defn}\label{defn:AV}
For any $j \in \Z_{\leq 0}$, 
we define an isomorphism 
\[
\theta^{j}_{K, S} : 
\varepsilon_{1-j}\Xi_{K/k, S} (1-j)
\xrightarrow{\sim}
\varepsilon_{1-j} \Det_{\Z_p[G]}^{-1} (\Delta_{K, S} (j)) 
\otimes_{\varepsilon_{1-j}\Z_p[G]} 
\varepsilon_{1-j} \Xi_{K/k, S}^{\loc} (j) 
\]
by the following composite map
\begin{align*}
&\theta^{j}_{K, S} : 
\varepsilon_{1-j}\Xi_{K/k, S} (1-j)
:= 
\varepsilon_{1-j} \parenth{
\Det_{\Z_p[G]}^{-1} (\RG (G_{K, S} , \Z_p (j))) \otimes_{\Z_p [G]} 
\Det_{\Z_p[G]}^{-1} (X_K (j-1)^+) }
\\
&\overset{\eqref{triangle}}{\xrightarrow{\sim}}
%\varepsilon_{1-j} \parenth{
%\Det_{\Z_p[G]}^{-1} \Big{(}\bigoplus_{v \in S_f}\RG(K_v, \Z_p (j))\Big{)}
%\otimes 
%\Det_{\Z_p[G]}^{-1} (\Delta_{K, S} (j)) 
%\otimes 
%\Det_{\Z_p[G]}^{-1} (X_K (j-1)^+) 
%}
%\\
%&\overset{\psi}{\xrightarrow{\sim}}
\varepsilon_{1-j} \parenth{
\Det_{\Z_p[G]}^{-1} (\Delta_{K, S} (j)) 
\otimes 
\Big{(}\bigotimes_{v \in S_f}
\Det_{\Z_p[G]}^{-1} (\RG(K_v, \Z_p (j)) )\Big{)}
\otimes 
\Det_{\Z_p[G]}^{-1} (X_K (j-1)^+) 
}
\\
&\overset{\eqref{1/2}}{\xrightarrow{\sim}}
\varepsilon_{1-j} \parenth{
\Det_{\Z_p[G]}^{-1} (\Delta_{K, S} (j)) 
\otimes \Big{(}\bigoplus_{v \in S_f}
\Det_{\Z_p[G]}^{-1} (\RG(K_v, \Z_p (j)))\Big{)}
\otimes 
\Det_{\Z_p[G]}^{-1} (X_K (j)) 
}
\\
&=
\varepsilon_{1-j} \Det_{\Z_p[G]}^{-1} (\Delta_{K, S} (j)) 
\otimes_{\varepsilon_{1-j}\Z_p[G]} 
\varepsilon_{1-j} \Xi_{K/k, S}^{\loc} (j),
\end{align*}
%\MEMO{$\psi$ の表記を全て消す}
where the second isomorphism is induced by 
\begin{equation}\label{1/2} 
\varepsilon_{1-j} X_K (j-1)^+ \overset{\eqref{canonical_id}}{\simeq}
 \varepsilon_{1-j} (X_K (1-j)^\ast) 
\overset{(\beta_K^{1-j})^{\ast}}{\xrightarrow{\sim}}
 \varepsilon_{1-j} (X_K (1-j)^{+, \ast})  
= \varepsilon_{1-j} (X_K (1-j)^\ast)
\overset{\eqref{Xpair}}{\simeq}
\varepsilon_{1-j}X_K (j). 
\end{equation}
Here we write $(-)^\ast := \Hom_{\Z_p} (- , \Z_p) $. 
\end{defn}

For any subfield $k \subset M \subset K$ and any $j \in \Z_{\geq 1}$,
there are natural identifications 
\[
\Z_p [\Gal (M/k)] \otimesL_{\Z_p[G]} \RG (G_{K, S} , \Z_p (j)) = \RG (G_{M, S} , \Z_p (j))
\]
and $\Z_p [\Gal (M/k)] \otimes_{\Z_p[G]} X_K (j) = X_M (j)$. 
These identifictions induce a canonical transition map 
$\Xi_{K/k , S} (j) \twoheadrightarrow  \Xi_{M/k , S} (j)$. 
Similarly, we obtain canonical transition maps 
$\Det_{\Z_p[G]}^{-1} (\Delta_{K, S} (j)) \twoheadrightarrow \Det_{\Z_p[\Gal (M / k)]}^{-1} (\Delta_{M, S} (j)) $ and 
$\Xi_{K/k, S}^{\loc} (j) \twoheadrightarrow \Xi_{M/k, S}^{\loc} (j)$. 

\begin{lem}\label{lem:theta_func}
Let $M$ be a CM-field such that $k \subset M \subset K$. 
Then, there is the following commutative diagram 
\[\xymatrix{
\varepsilon_{1-j} \Xi_{K/k, S} (1-j)
\ar[r]_-{\sim}^-{\theta^{j}_{K, S} }
\ar@{->>}[d]
&
\varepsilon_{1-j} \Det_{\Z_p[G]}^{-1} (\Delta_{K, S} (j)) 
\otimes_{\varepsilon_{1-j}\Z_p[G]} 
\varepsilon_{1-j}\Xi_{K/k, S}^{\loc} (j) 
\ar@{->>}[d]
\\
\varepsilon_{1-j}\Xi_{M/k, S} (1-j)
\ar[r]_-{\sim}^-{\theta^{j}_{M, S} }
&
\varepsilon_{1-j} \Det_{\Z_p[\Gal (M/k)]}^{-1} (\Delta_{M, S} (j)) 
\otimes_{\varepsilon_{1-j}\Z_p[\Gal (M/k)]} 
\varepsilon_{1-j} \Xi_{M/k, S}^{\loc} (j)  
}
\]
for any $j \in \Z_{\leq 0}$, 
where both vertical arrows are induced by the canonical transition maps. 
\end{lem}
\begin{proof}
In Definition \ref{defn:AV}, 
we define the isomorphism $\theta_{K, S}^j$ using 
the exact triangle \eqref{triangle} and the isomorphism \eqref{1/2}. 
Clearly, these behave functorially with respect to change of fields. 
Thus, we know the commutativity of this diagram. 
\end{proof}

The following theorem is a special case of \cite[\S 5.1, Lem. 17]{BuFl01}. 

\begin{thm}\label{thm: comm of period reg}
For any $j \in \Z_{\leq 0}$, the following diagram is commutative: 
% up to sign $(-1)^{[k:\Q]}$; 
\[
\xymatrix{
\C_p \otimes_{\Z_p} \varepsilon_{1-j} \Xi_{K/k, S} (1-j)  \ar[r]^-{\theta_{K, S}^{ j}}_-{\simeq}
\ar[d]^-{\vartheta_{K/k, S}^{1-j}}_-{\simeq}
&\C_p\otimes_{\Z_p} \varepsilon_{1-j} \Det_{\Z_p[G]}^{-1} (\Delta_{K, S} (j)) \otimes_{\varepsilon_{1-j}\Z_p[G]} 
\varepsilon_{1-j} \Xi_{K/k, S}^{\loc} (j) 
\ar[d]^-{\vartheta_{K/k, S}^{j} \otimes \Phi_{K / k, S}^{\loc, j}}_-{\simeq} \\
\varepsilon_{1-j}\C_p[G] & \varepsilon_{1-j} ( \C_p[G] \otimes_{\C_p[G]} \C_p[G]) 
\ar[l]_-{\simeq}^-{ab \mapsfrom a \otimes b}. 
}
\]
\end{thm}
\begin{proof}
For simplicity, we write $\varepsilon:= \varepsilon_{1-j}$. 
%$\RG^{\loc} (\Z_p (j)) := \bigoplus_{v \in S_f} \RG (K_v , \Z_p (j))$, $\varepsilon:= \varepsilon_{1-j}$, 
%$\Det_{R} (-) := [-]_R$ and  $\Det_{R} (-) := [-]_R^{-1}$. 
%Then, 
We consider the following commutative diagram:  
\[
\xymatrix{
 \varepsilon \Det_{\Q_p[G]}^{-1} (\RG (G_{K, S} , \Q_p(j)))
\ar[r]^-{\eqref{triangle}}_-{\simeq}
\ar[d]_-{\simeq}
&
\varepsilon
\Big{(}
\displaystyle
\Det_{\Q_p[G]}^{-1} (\Q_p \otimesL_{\Z_p} \Delta_{K, S} (j))
\otimes
\bigotimes_{v \in S_f} \Det_{\Q_p[G]}^{-1}(\RG(K_v, \Q_p (j)))
\ar[d]^-{\vartheta_{K/k, S}^{j} \otimes \eqref{dual_version}}_-{\simeq}
\Big{)}
\\
 \varepsilon \Det_{\Q_p[G]}(H^1 (G_{K, S} , \Q_p(j)))
\ar[r]^-{\eqref{Lemma(3)}}_-{\simeq}
&
\varepsilon
\Det_{\Q_p[G]}((\Q_p \otimes_\Q K)^\ast). 
}
\]
When $j <0$, the commutativity of this diagram is clear by the definitions of the maps in the diagram. 
%of $\vartheta_{K/k, S}^{j}, \vartheta_{K/k, S}^{1-j}$ and $ \Phi_{K/k, S}^{\loc, j} $. 
When $j =0$, this commutativity is obtained by applying Lemma \ref{lem:nine_term} for the diagram \eqref{diagram2}. 
%Here we note that since the grade of $\varepsilon_{1-j}[\Delta_{K, S} (j)]^{-1}_{\Z_p[G]}$ is zero, we have not to care an effect $\psi$ on signs. 

%We write $\Q_p \cdot (-) :=\Q_p \otimes_{\Z_p} (-)$ and $[-]:= [-]_{\Q_p[G]}$ for simplicity. 
On the other hand, we also have the following commutative diagram:

{\footnotesize
\[
\xymatrix@C=30pt{
\C_p \otimes
\varepsilon
\parenth{
\D_{\Q_p[G]}((\Q_p \otimes_\Q K)^\ast)
\otimes
\D_{\Z_p[G]}^{-1}(X_K(j-1)^+)
}
\ar[r]_-{\simeq}^-{\id \otimes \eqref{1/2}}
\ar[d]^-{(\gamma_K^{1-j})^{\ast, -1} \otimes \id}_-{\simeq}
&
\C_p \otimes
\varepsilon
\parenth{
\D_{\Q_p[G]}((\Q_p \otimes_\Q K)^\ast)
\otimes
\D_{\Z_p[G]}^{-1}(X_K(j))
}
\ar[d]^-{(\alpha_K^{1-j})^{\ast, -1} \otimes \eqref{Xpair}}_-{\simeq}
\\
\C_p \otimes
\varepsilon
\parenth{
\D_{\Z_p[G]}(X_K(1-j)^{\ast})
\otimes
\D_{\Z_p[G]}^{-1}(X_K(j-1)^+)
}
\ar[r]^-{\id \otimes \eqref{canonical_id}}_-{\simeq}
\ar[d]^-{\eqref{canonical_id} \otimes \id}_-{\simeq}
&
\C_p \otimes
\varepsilon
\parenth{
\D_{\Z_p[G]}(X_K(1-j)^\ast)
\otimes
\D_{\Z_p[G]}^{-1}(X_K(1-j)^\ast)
}
\ar[d]^-{\ev}_-{\simeq}
\\
\C_p \otimes
\varepsilon
\parenth{
\D_{\Z_p[G]}(X_K(j-1)^+)
\otimes
\D_{\Z_p[G]}^{-1}(X_K(j-1)^+)
}
\ar[r]^-{\ev}_-{\simeq}
&
\varepsilon
\C_p[G], 
}
\]}
where we write $\D_R(-):= \Det_R(-)$ and 
we identify $\varepsilon_{1-j} (X_K (1-j)^{+, \ast}) = \varepsilon_{1-j} (X_K (1-j)^{\ast})$. 
The commutativity of the upper square follows from the definition of \eqref{1/2} and $\alpha_{K}^{1-j} = \beta_K^{1-j} \circ \gamma_K^{1-j}$. 
The commutativity of the bottom square is clear. 

%From these two diagram and the definition of the map $\vartheta_{K/k, S}^{1-j},  \Phi_{K/k, S}^{¥loc, j} $ and $\theta_{K, S}^j$, 
By Definition \ref{def:period_reg>0}, we know that $\vartheta_{K/k, S}^{1-j}$ coincides with the map which is obtained by combining 
the left and bottom arrows in the first diagram with the all left vertical arrows and the bottom horizontal arrow in the second diagram. 
On the other hand, by the Definition \ref{def:vartheta_loc}, \ref{def:period_reg<0} and \ref{defn:AV}, 
the composite map $(\vartheta_{K/k, S}^{j} \otimes \Phi_{K/k, S}^{\loc, j} ) \circ \theta_{K, S}^j$ coincides with the map which is obtained by combining 
the upper and right arrows in the first diagram with the all right vertical arrows in the second diagram. 
Therefore, we obtain the commutativity of the diagram in Theorem \ref{thm: comm of period reg}. 
\end{proof}

%%%%%%%%%%%%%%%%%%%%%%%%%%%%%%%%%%%%%%%%%%%%%%%
\section{Iwasawa theory over a totally real field}\label{s:zeta<1}
%%%%%%%%%%%%%%%%%%%%%%%%%%%%%%%%%%%%%%%%%%%%%%%
In this section, we study classical Iwasawa theory over a totally real field. 
A main result in this section is Theorem \ref{thm:zeta<1}. 
This theorem combines the Iwasawa main conjecture and the minus component of the equivariant Tamagawa number conjecture proved by Dasgupta--Kakde--Silliman \cite{DKS}. 

%%%%%%%%%%%%%%%%%%%%%%%%%%%%%%%%%%%%%%%%%%%%%%%
\subsection{The $L$-functions}\label{ss:L-function}
%%%%%%%%%%%%%%%%%%%%%%%%%%%%%%%%%%%%%%%%%%%%%%%

We introduce $L$-functions. 

Let $K/k$ be a finite abelian CM-extension and put $G:= \Gal (K/k)$. 
For a finite prime $v$ of $k$, we define $N(v)$ as the cardinality of the residue field of $k$ at $v$.
For any $\C$-valued character $\chi$ of $G$ and any finite set $S$ of places of $k$, 
we define the $L$-function $L_S(\chi, s)$ by
\[
L_S(\chi, s) 
= \prod_{\overset{v \nmid \ff_{\chi}}{v \notin S}} \parenth{1 - \chi(v)N(v)^{-s}}^{-1},
\]
where $v$ runs over the finite primes of $k$ 
that neither divide the conductor $\ff_{\chi}$ of $\chi$ nor are contained in $S$ and 
$\chi (v)$ is the image of the {\it arithmetic} Frobenius at $v$ under $\chi$. 
It is well-known that this infinite product converges absolutely for 
$s \in \C$ whose real part is larger than $1$, 
and has a meromorphic continuation to the whole complex plane $s \in \C$. 
When $S$ is the empty set, we write $L(\chi, s):= L_{\emptyset}(\chi, s)$. 

The equivariant $L$-function is defined by 
\[
\Theta_{K/k, S} (s) := \sum_{\chi} L_S(\chi, s) e_\chi ,
%\in \C[\GG_L] (\simeq \C_p [\GG_L]), 
\]
where $\chi$ runs over all characters of $G$ and $e_\chi := \frac{1}{\# G}\sum_{\sigma \in G} \chi (\sigma) \sigma^{-1}$ is the associated idempotent. 
This is a $\C[G]$-valued function. If $S$ is empty, we write $\Theta_{K/k} (s) := \Theta_{K/k, \emptyset} (s)$ for simplicity. 

\begin{rem}\label{rem:L-value}
Let $c \in G$ be the complex conjugation. 
For any integer $j$, we set an idempotent 
\[
\varepsilon_j:= \frac{1+(-1)^jc}{2} \in \Q[G]
\] as in the previous section. 
We note that $L_S (\chi, s)$ has a pole only if $\chi$ is a trivial character $\trivial$, in which case $L_S ({\trivial}, s)$ has a simple pole at $s=1$. 
Therefore, we can always define $\Theta_{K/k, S} (j) \in \C[G]$ for any $j \in \Z \setminus \{ 1 \}$ and 
$
\varepsilon_1 \Theta_{K/k, S} (1) = \sum_{\chi (c) =-1} L_S(\chi, s) e_\chi \in \C [G]. 
$ 
Then, we have 
\[
\varepsilon_j \Theta_{K/k, S} (j) \in (\varepsilon_j \C[G])^\times
\]
for any $j \in \Z_{\geq 1}$. 
By this fact and the functional equation, it also follows that 
\[
\varepsilon_{1-j} \Theta_{K/k, S} (j) \in (\varepsilon_{1-j} \C[G])^\times
\]
for any integer $j< 0$ and $\varepsilon_{1} \Theta_{K/k} (0) \in (\varepsilon_{1} \C[G])^\times$. 
\end{rem}

We write $L_S^\ast(\chi, 0) \in \C^\times $ 
for the leading term of the Taylor expansion of  $L_S (\chi, s)$ at $s=0$. 
We also define the leading term of the equivariant $L$-function at $s=0$ by  
\[
\Theta_{K/k, S}^\ast (0) := \sum_{\chi} L_S^\ast(\chi, 0) e_\chi 
\in \C[G]^{\times}.
\]

%%%%%%%%%%%%%%%%%%%%%%%%%%%%%%%%%%%%%%%%%%%%%%%%%%%%%%%
\subsection{The minus component of the equivariant Tamagawa number conjecture}\label{ss:ETNCminus}
%%%%%%%%%%%%%%%%%%%%%%%%%%%%%%%%%%%%%%%%%%%%%%%%%%%%%%%

In this subsection, we review 
the minus component of the equivariant Tamagawa number conjecture for $\mathbb{G}_m$ (ETNC$^-$) which 
is proved by Dasgupta--Kakde--Silliman \cite{DKS}. 

We fix an odd prime number $p$ and also fix an isomorphism $\C\simeq \C_p$ and identify them as in the previous section. 
For simplicity, we write $\varepsilon^- := \varepsilon_1 = \frac{1-c}{2} \in \Z_p[G]$. 
We set $\Z_p[G]^- := \varepsilon^-\Z_p [G]$ and 
$M^- := \varepsilon^- M$ for any $\Z_p[G]$-module $M$. 

Let $S$ be a finite set of places of $k$ which contains $S_\infty \cup S_p \cup S_{\ram} (K/k)$
as in the Introduction and set $S_f := S \setminus S_\infty$. 

We define an isomorphism 
\[
\widetilde{\vartheta_{K/k, S}^{0}} : \C_p \otimes_{\Z_p}\Det_{\Z_p[G]}^{-1} (\RG(G_{K, S} , \Z_p (1)))^- \xrightarrow{\sim} \C_p[G]^-
\]
by the following composite map 
\begin{align*}
&\widetilde{\vartheta_{K/k, S}^{0}} : \C_p \otimes_{\Z_p}\Det_{\Z_p[G]}^{-1} (\RG(G_{K, S} , \Z_p (1)))^-
\\
&\overset{\eqref{Q_p(1)}}{\xrightarrow{\sim}} 
\C_p \otimes_{\Q_p} \Det_{\Q_p[G]} (\Q_p \otimes_{\Z} \OO_{K, S}^\times)^{-}  
\otimes_{\Q_p[G]^-} \Det_{\Q_p[G]}^{-1} (\bigoplus_{w \in S_{f,K}} \Q_p)^{-} 
\\
&\overset{\reg_S}{\xrightarrow{\sim}}
\C_p \otimes_{\Q_p} \Det_{\Q_p[G]} (\bigoplus_{w \in S_{f,K}} \Q_p)^{-}  
\otimes_{\Q_p[G]^-} \Det_{\Q_p[G]}^{-1} (\bigoplus_{w \in S_{f,K}} \Q_p)^{-} 
\overset{\ev}{\xrightarrow{\sim}} \C_p[G]^-, 
\end{align*}
where the second isomorphism is induced by the minus-component of the Dirichlet regulator isomorphism: 
\[
\reg_S : (\C_p \otimes_{\Z} \OO_{K, S}^\times)^- \xrightarrow{\sim} ( \bigoplus_{w \in S_{f,K}} \C_p )^-
\; ; \; x \mapsto (-\log |x|_w)_{w \in S_{f, K}}. 
\]

Then, the ETNC$^-$ proved by Dasgupta--Kakde--Silliman \cite{DKS} is the following. 
\begin{thm}[Theorem 1 in \cite{DKS}]\label{thm:ETNC-}
There is a (unique) $\Z_p[G]^-$-basis 
\[
\widetilde{z^{0,-}_{K/k, S}} \in \Det_{\Z_p[G]}^{-1}(\RG(G_{K, S} , \Z_p (1)))^- \Big{(} \subset \C_p \otimes_{\Z_p}  \Det_{\Z_p[G]}^{-1}(\RG(G_{K, S} , \Z_p (1)))^- \Big{)}
\]
such that 
\[
\widetilde{\vartheta_{K/k, S}^{0}} (\widetilde{z^{0,-}_{K/k, S}}) = \varepsilon^-\Theta_{K /k , S}^\ast(0)^\#
\]
in $\C_p[G]^-$, where $(-)^\#$ denotes the image of $(-)$ under the involution defined in \S \ref{ss:alg_notation}. 
\end{thm}
\begin{rem}
The statement of \cite[Theorem 1]{DKS} is different from Theorem \ref{thm:ETNC-}. 
However, both are equivalent (cf.~\cite[Proposition 3.5 and its proof]{Kur21} and \cite[Proposition 3.4]{BKS16}). 
\end{rem}

We may reformulate Theorem \ref{thm:ETNC-} in terms of the complex $\Delta_{K, S} (0)$ and the isomorphism $\vartheta_{K/k, S}^{0}$ defined in Definition \ref{def:period_reg<0} as follows.

By the global duality (cf.~\cite[Proposition 5.4.3]{Nek06}), 
we have a canonical quasi-isomorphism 
\begin{equation}
\Delta_{K, S} (0) \simeq \RHom_{\Z_p} (\RG (G_{K,S} , \Z_p (1)) , \Z_p)[-3]. 
\end{equation}
This quasi-isomorphism together with \eqref{homRRG} induces an isomorphism
\begin{equation}\label{duality_j=1}
\Det_{\Z_p[G]}^{-1} (\Delta_{K, S} (0))^- \simeq 
%\Det_{\Z_p[G]} (\RG (G_{K,S} , \Z_p (1))^{-, \vee, \#}= 
\Det_{\Z_p[G]}^{-1} (\RG (G_{K,S} , \Z_p (1)))^{-, \#}. 
\end{equation}
%where the last equality follows from the fact that the grade of $\Det_{\Z_p[G]}^{-1} (\RG (G_{K,S} , \Z_p (1)))^-$ is zero 
%(see the definition of $(-)^\vee$ in \S \ref{App:det}). 
Using this, we define a $\Z_p[G]^-$-basis 
\[
z^{0,-}_{K/k, S} \in \Det_{\Z_p[G]}^{-1} (\Delta_{K, S} (0))^-
\]
as the image of $\widetilde{z^{0,-}_{K/k, S}} $ under the inverse of \eqref{duality_j=1}. 

For any finite prime $v$ of $k$, 
we set 
\begin{equation}\label{delta}
\delta_{v} (K/k):= 1 - \sigma_{v} \cdot \sum_{\tau \in I_v} \tau + \frac{1}{\# (D_v / I_v)} \sum_{\tau \in D_v} \tau
\in \Q[G]^\times, 
\end{equation}
where $D_v$ and $I_v$ are the decomposition group and the inertia group in $G$ respectively, 
and $ \sigma_v \in D_v / I_v$ is the arithmetic Frobenius.

\begin{lem}\label{lem:ETNC^-}
The basis $z^{0,-}_{K/k, S} \in \Det_{\Z_p[G]}^{-1} (\Delta_{K, S} (0))^- $ satisfies the following. 
\begin{itemize}
\item[(1)] We have 
\[
\vartheta_{K/k, S}^{0} (z^{0,-}_{K/k, S}) = \varepsilon^- \parenth{\prod_{v \in S_f} \delta_v (K / k)} \Theta_{K / k } (0)
\]
in $\C_p[G]^-$. 
\item[(2)] For any finite CM abelian extension $K' / k$ such that  $K \subset K'$,  
 the canonical transition map 
 \[
 \Det_{\Z_p[\Gal(K'/k)]}^{-1} (\Delta_{K', S} (0))^- \twoheadrightarrow \Det_{\Z_p[G]}^{-1} (\Delta_{K, S} (0))^-
 \]
sends  the basis $z^{0,-}_{K'/k, S} $ to $z^{0,-}_{K/k, S} $. 
\end{itemize}
\end{lem}
\begin{proof}
(1) We consider the following commutative diagram: 
\[
\xymatrix@C=70pt{
(\C_p \otimes_{\Z} \OO_{K, S}^\times)^- 
\ar[r]^-{\oplus_{w \in S_{f,K}}\ord_{w}}_{\sim}
\ar@{=}[d]
&( \bigoplus_{w \in S_{f,K}} \C_p )^-
\ar[d]^-{f}_-{\simeq}
\\
(\C_p \otimes_{\Z} \OO_{K, S}^\times)^- 
\ar[r]^-{\reg_S}_{\sim}
&( \bigoplus_{w \in S_{f,K}} \C_p )^-, 
}
\]
where the right vertical isomorphism is given by
\[
f : (\bigoplus_{w \in S_{f,K}} \C_p)^- \xrightarrow{\sim} (\bigoplus_{w \in S_{f,K}} \C_p)^- 
; \;\; (x_w)_{w \in S_{f,K}} \mapsto (\log (N (w)) \cdot x_w)_{w \in S_{f,K}}. 
\]
Therefore, by the definitions of the maps $\widetilde{\vartheta_{K/k, S}^{0}}$ and $\vartheta_{K/k, S}^{0}$, 
and Lemma \ref{lem:det_easy} below, 
we have
\begin{align*}
\vartheta_{K/k, S}^{0} (z^{0,-}_{K/k, S})  
&= 
\Big{(} {\det}_{\C_p[G]^-} (f | (\bigoplus_{w \in S_{f,K}} \C_p)^- )\Big{)}^{-1} \cdot 
\widetilde{\vartheta_{K/k, S}^{0}} (\widetilde{z^{0,-}_{K/k, S}})^\#
\\
&=
\Big{(} \sum_{\chi \in \widehat{G}} \prod_{\underset{\chi (v) =1}{v \in S_f}} \log (N(w_v))  e_\chi \Big{)}^{-1} 
\cdot \varepsilon^- \Theta_{K/k , S}^\ast (0) 
= \varepsilon^- \parenth{\prod_{v \in S_f} \delta_v (K / k)} \Theta_{K / k } (0), 
\end{align*}
where $w_v$ is a prime of $K$ above $v$ and 
in the product of the third term, $v$ runs all primes in $S_f$ such that $\ker \chi$ contains 
the decomposition group of $v$ in $G$. 

(2) We consider the following commutative diagram: 
\[
\xymatrix{
(\C_p \otimes_{\Z} \OO_{K', S}^\times)^-
 \ar[r]^-{\reg_S}_-{\sim} 
\ar[d]^-{\sim}_-{N_{K' / K}}
&( \bigoplus_{w' \in S_{f,K'}} \C_p )^- 
\ar[d]_-{\sim}^-{(a_{w'})_{w'} \mapsto (\sum_{w'|w} a_{w'})_w}
\\
(\C_p \otimes_{\Z} \OO_{K, S}^\times)^- \ar[r]^-{\reg_S}_-{\sim} 
&( \bigoplus_{w \in S_{f,K}} \C_p )^-, 
}
\]
where the left vertical arrow is the norm map. 
From this, we obtain the following commutative diagram: 
\[
\xymatrix{
 \Det_{\Z_p[\Gal(K' / k)]}^{-1} (\RG(G_{K', S} , \Z_p (1)))^- 
 \ar@{^{(}->}[r]^-{\widetilde{\vartheta_{K'/k, S}^{0}}}
 \ar@{->>}[d]
 &\C_p[\Gal(K' / k)]^-
 \ar@{->>}[d]_-{\sim}^-{\res_{K' / K}}
 \\
 \Det_{\Z_p[G]}^{-1} (\RG(G_{K, S} , \Z_p (1)))^-
\ar@{^{(}->}[r]^-{\widetilde{\vartheta_{K/k, S}^{0}}}
&
\C_p[G]^-, 
}
\]
where the left vertical arrow is the canonical transition map and 
right vertical arrow is the restriction map. 
Since $\res_{K' / K} (\varepsilon^- \Theta_{K' /k , S}^\ast(0)) = \varepsilon^-\Theta_{K /k , S}^\ast(0)$, 
we know that the left vertical arrow sends $\widetilde{z^{0,-}_{K'/k, S}}$ to $\widetilde{z^{0,-}_{K/k, S}}$. 
Since the isomorphism \eqref{duality_j=1} is compatible with the canonical transition maps, 
the claim (2) follows from this. 
\end{proof}

\begin{lem}\label{lem:det_easy}
Let $R$ be a commutative ring, $P$ a finitely generated projective $R$-module with (locally constant) rank $r$. 
Suppose that we are given an automorphism of $R$-module $f : P \xrightarrow{\sim} P$. 
Then, the isomorphism 
\[
\widetilde{\vartheta} : \Det_R (P) \otimes_R \Det_R^{-1} (P) \overset{f \otimes \id}{\xrightarrow{\sim}} 
\Det_R (P) \otimes_R \Det_R^{-1} (P) \overset{\ev}{\xrightarrow{\sim}} R
\]
coincides with the $\det_R (f|P)$ times the evaluation map 
$\Det_R (P) \otimes_R \Det_R^{-1} (P) \overset{\ev}{\xrightarrow{\sim}} R$.  
\end{lem}
\begin{proof}
Since we can consider for each connected component on $\Spec (R)$, 
we may assume that $P$ is a free $R$-module with rank $r \in \Z_{\geq 1}$.  
Take a basis $\{ x_1, \dots , x_r\}$ of $P$. 
Then, we have 
\begin{align*}
\widetilde{\vartheta} ( (x_1 \wedge \cdots \wedge x_r) \otimes (x_1 \wedge \cdots x_r)^\ast ) 
&= \ev ( (f(x_1) \wedge \cdots \wedge f(x_r)) \otimes (x_1 \wedge \cdots x_r)^\ast )
\\
&=
\ev ( {\det}_R (f | P)(x_1 \wedge \cdots \wedge x_r) \otimes (x_1 \wedge \cdots x_r)^\ast )
\\
&=  {\det}_R (f | P). 
\end{align*}
From this, the claim is clear. 
\end{proof}

%%%%%%%%%%%%%%%%%%%%%%%%%%%%%%%%%%%%%%%%%%%%%%%
\subsection{Zeta element for $\Z_p (j)$ when $j \leq 0$}\label{ss:main_conj}
%%%%%%%%%%%%%%%%%%%%%%%%%%%%%%%%%%%%%%%%%%%%%%%

In this subsection, we give a main result of \S \ref{s:zeta<1}. 

Let $K/k$ be a finite abelian CM-extension and $S$  a finite set of places of $k$ which contains $S_\infty \cup S_p \cup S_{\ram} (K/k)$
as in the previous section. We put $S_f := S \setminus S_\infty$, $K_\infty := \cup_{n \geq 0} K (\mu_{p^n})$, and 
$\GG_{K_{\infty}} := \Gal (K_\infty / k)$. 
We also set 
\[
\Omega(K_\infty / k) := 
\{ k \subset L \subset K_{\infty} \;\;| \;\; [L:k]<\infty , \;L: \text{CM-field} \}. 
\]
For any $L \in \Omega(K_\infty / k)$, 
we put $\GG_L := \Gal (L/k)$ and $\Lambda := \Z_p[[\GG_{K_\infty}]]$. 
Let $c \in \GG_{K_\infty}$ be the complex conjugation. 
For any $j \in \Z$, we consider an idempotent
\[\varepsilon_j := \frac{1 + (-1)^jc}{2} \in \Lambda. 
\]
For any $ L \in \Omega(K_\infty / k) $, 
we also write $\varepsilon_j \in \Z_p[\GG_L]$ for the natural image of $\varepsilon_j$ under the restriction map 
$\Lambda \twoheadrightarrow \Z_p[\GG_L]$. 
For any $m \in \Z$, 
we consider the ring automorphism 
\[
\twist_m : \Lambda \xrightarrow{\sim} \Lambda ; \sigma \mapsto \chi_{\cyc} (\sigma)^m \sigma
\] induced by the $m$-th power of the cyclotomic character $\chi_{\cyc} : \GG_{K_\infty} \to \Z_p^\times$. 
Then, this induces a ring isomorphism 
$
\twist_m : \varepsilon_j \Lambda \xrightarrow{\sim}  \varepsilon_{j+m} \Lambda 
$
for any $j, m \in \Z$ and we also write $\twist_m$ for this. 

%\begin{defn}\label{defn:Delta_Iw}
For any $j \in \Z $, 
we consider a complex of $\Lambda$-modules 
$
\Delta_{K_\infty, S} (j) 
$
defined by the projective limit of $ \Delta_{L, S} (j)$ for $L \in \Omega(K_\infty / k)$, 
where $\Delta_{L, S} (j)$ is defined in \S \ref{ss:period reg map}. 
%\end{defn}

A main result in \S \ref{s:zeta<1} is the following. 

\begin{thm}\label{thm:zeta<1}
For any $j \in \Z_{\leq 0}$,  
there is a (unique) $\varepsilon_{1-j} \Lambda$-basis 
\[
\varepsilon_{1-j} Z_{K_\infty / k, S}^{j} \in \varepsilon_{1-j} \Det_{\Lambda}^{-1} (\Delta_{K_\infty, S} (j)) 
\]
such that the following statements hold. 
\begin{itemize}
\item[(1)]
For any $L \in \Omega(K_\infty / k) $,  
the following composite map 
\[
\varepsilon_{1-j} \Det_{\Lambda}^{-1} (\Delta_{K_\infty, S} (j)) {\twoheadrightarrow} 
\varepsilon_{1-j}\Det_{\Z_p[\GG_L]}^{-1} (\Delta_{L, S} (j)) \xrightarrow{\vartheta_{L/k, S}^{j}} 
\varepsilon_{1-j} \C_p[\GG_{L}] 
\]
sends $\varepsilon_{1-j} Z_{K_\infty / k, S}^{j}$ to 
\[
\begin{cases}
\varepsilon_{1-j} \Theta_{L / k , S} (j)  & \text{if } j <0, 
\\
\varepsilon_{1} \parenth{\prod_{v \in S_f} \delta_v (L / k)} \Theta_{L / k } (0) 
& \text{if } j =0, 
\end{cases}
\]
where the first surjective map is induced by the natural identification 
$\Z_p[\GG_{L}] \otimesL_{\Lambda} \Delta_{K_\infty, S} (j)
=
\Delta_{L, S} (j) $ and the second map $\vartheta_{L/k, S}^{j}$ is defined in Definition \ref{def:period_reg<0}. 
\item[(2)]
For any non-positive integers $j$ and $j'$, 
we have 
\[
\Twist_{j, j'}^\Delta ( \varepsilon_{1-j}Z_{K_\infty / k, S}^{j})  = \varepsilon_{1-j'} Z_{K_\infty / k, S}^{j'}, 
\]
where $\Twist_{j, j'}^\Delta$ is the $\twist_{j-j'}$-semilinear isomorphism defined in \S \ref{ss:twist_CM}.
\end{itemize}
\end{thm}
By Remark \ref{rem:L-value}, we note that the $L$-values which appear in Theorem \ref{thm:zeta<1} (1) 
are non-zero divisors in $\varepsilon_{1-j} \C_p[\GG_{L}] $. 

%%%%%%%%%%%%%%%%%%%%%%%%%%%%%%%%%%%%%%%%%%%%%%%
\subsection{Proof of Theorem \ref{thm:zeta<1}}\label{ss:prf_zeta<1}
%%%%%%%%%%%%%%%%%%%%%%%%%%%%%%%%%%%%%%%%%%%%%%%
In this subsection, we prove Theorem \ref{thm:zeta<1}. 

By Lemma \ref{lem:ETNC^-} (2), we can define a $\varepsilon_1 \Lambda$-basis
$\varepsilon_1 Z_{K_\infty / k, S}^{0}   \in \varepsilon_1 \Det_{\Lambda}^{-1} (\Delta_{K_\infty, S} (0))$ by 
\[
\varepsilon_1 Z_{K_\infty / k, S}^{0} := (z_{L/k, S}^{0, -})_{L \in \Omega(K_\infty / k) } 
\in 
\varprojlim_{L \in \Omega(K_\infty / k) } \Det_{\Z_p[\GG_L]}^{-1} (\Delta_{L, S} (0))^- 
\simeq
\varepsilon_1 \Det_{\Lambda}^{-1} (\Delta_{K_\infty, S} (0)).  
\]

On the other hand, we consider the $p$-adic $L$-function. 
Let $Q$ be the total fractional ring of $\Lambda$. 
We set 
\[
\Lambda^\sim := \{x \in Q | (\sigma -1)x \in \Lambda
\text{ for any } \sigma \in \GG_{K_\infty} \}.
\]
For any $L \in \Omega(K_\infty / k)$, we write $\res_{K_\infty / L} : \Lambda \to \Z_p[\GG_L]$ for the restriction map. 
Note that for any $j \in \Z_{\leq 0}$ and $L \in \Omega(K_\infty / k)$, 
the composite map $\res_{K_\infty / L} \circ \twist_{1-j} : \varepsilon_0 \Lambda \to \Z_p[\GG_L]$ can be 
extended to $\res_{K_\infty / L} \circ \twist_{1-j} : \varepsilon_0 \Lambda^\sim \to \Q_p[\GG_L]$. 
%Then, for any character $\chi$ of $\GG_{K_\infty}$ with finite order and $j \in \Z$, 
%we note that the composite map $\chi \circ \twist_j : \Lambda^\sim \to \C_p$ is well-defined 
%except the case when $j=0$ and $\chi$ is trivial. 
Then, there exists a canonical element 
\[
\cL_{K_\infty /k , S} \in \varepsilon_0  \Lambda^{\sim}
\]
(the Deligne--Ribet $p$-adic $L$-function) that satisfies 
\begin{equation}\label{interpolation property}
\res_{K_\infty / L} (\twist_{1-j} (\cL_{K_\infty / k , S})) = \varepsilon_{1-j} \Theta_{L/k, S} (j) 
\end{equation}
for any $L \in \Omega(K_\infty / k)$ and $j \in \Z_{\leq 0}$. 

\begin{prop}\label{prop:zeta=p_adic_L} 
The complex $\varepsilon_1 Q \otimesL_{\varepsilon_1\Lambda} \varepsilon_1 \Delta_{K_\infty, S} (0)$ is acyclic and the following canonical map
\[
\varepsilon_1 \Det_{\Lambda}^{-1} (\Delta_{K_\infty, S} (0)) \hookrightarrow 
\varepsilon_1Q \otimes_{\varepsilon_1 \Lambda} \varepsilon_1 \Det_{\Lambda}^{-1} (\Delta_{K_\infty, S} (0))
\xrightarrow{\sim} \varepsilon_1 Q 
\]
sends the basis $\varepsilon_1 Z_{K_\infty / k, S}^{0} \in \varepsilon_1 \Det_{\Lambda}^{-1} (\Delta_{K_\infty, S} (0))$ 
to $\twist_{1} (\cL_{K_\infty / k , S}) \in \varepsilon_1Q$. 
\end{prop}
\begin{proof}
The acyclicity of the complex $\varepsilon_1 Q \otimesL_{\varepsilon_1\Lambda} \varepsilon_1 \Delta_{K_\infty, S} (0)$ is well known (cf.~\cite[Theorem 3.7]{BS20functional}). 

Let $Z \in \varepsilon_1 Q$ be the image of $\varepsilon_1 Z_{K_\infty / k, S}^{0}$ under the composite map in the proposition. 
For any character $\chi$ of $\GG_{K_\infty}$ with finite order, 
we consider the map $ \varepsilon_1 Q \to \C_p \cup \{\infty\}$ which is induced by $\chi$ and 
also write $\chi$ for this. 
To prove the equality $Z = \twist_{1} (\cL_{K_\infty / k , S})$, it is enough to show that $\chi (Z) =\chi (\twist_{1} (\cL_{K_\infty / k , S}) )$ 
for almost all odd characters $\chi$ of $\GG_{K_\infty}$ with finite order. 

Let $D \subset \GG_{K_\infty}$ be the intersection of the decomposition groups of all $v \in S_f$ in $\GG_{K_\infty}$. 
Since no finite prime splits completely in $K_\infty / K$, we know that $[\GG_{K_\infty} : D] < \infty$. 
Let $\chi$ be an odd character of $\GG_{K_\infty}$ with finite order such that $\ker \chi \nsupseteq D$. 
We also set $L_\chi := K_\infty^{\ker \chi} \in \Omega (K_\infty / k)$, which is the CM-field attached to $\ker \chi$. 
We regard $\chi$ as a faithful character of $\GG_{L_\chi}$. 
Then, since no prime in $S_f$ splits completely in $L_\chi / k$ and $\chi$ is faithful, 
we know that the complex $\C_p^\chi \otimesL_{\Q_p[\GG_{L_\chi}]} \Delta_{L_\chi, S} (0)$ 
is acyclic by Lemma \ref{lem:complex} (2), where $\C_p^\chi = \C_p$ and $\GG_{L_\chi}$ acts this via $\chi$. 
In this case, the $\chi$-component of the isomorphism $\vartheta_{L_\chi, S}^{0, \chi}$ is the canonical one and 
the image of $\varepsilon_1 Z_{K_\infty / k, S}^{0}$ under the composite map 
\[
\varepsilon_1 \Det_{\Lambda}^{-1} (\Delta_{K_\infty, S} (0)) \to 
\C_p^\chi \otimesL_{\Q_p[\GG_{L_\chi}]} \Delta_{L_\chi, S} (0) 
\overset{\vartheta_{L_\chi, S}^{0, \chi}}{\xrightarrow{\sim}} \C_p
\]
coincides with $\chi (Z)$. 
Therefore, by Lemma \ref{lem:ETNC^-}, 
we have 
\[
\chi (Z) = \chi \Big{(}(\prod_{v \in S_f} \delta_v (L_\chi / k)) \Theta_{L_\chi / k } (0))\Big{)} 
= L_S (\chi, 0). 
\]
Therefore, by the interpolation property of the $p$-adic $L$-function \eqref{interpolation property}, 
we have $\chi (Z) = \chi (\twist_{1} (\cL_{K_\infty / k , S}) )$ for any odd character of $\GG_{K_\infty}$ with finite order such that $\ker \chi \nsupseteq D$. 
Since $[\GG_{K_\infty} : D] < \infty$, 
the number of odd characters of $\GG_{K_\infty}$ with finite order such that $\ker \chi \supseteq D$ is finite. 
This completes the proof of Proposition \ref{prop:zeta=p_adic_L}. 
\end{proof}
Finally, we prove Theorem \ref{thm:zeta<1}. 

\begin{proof}[Proof of Theorem \ref{thm:zeta<1}]
We already define the $\varepsilon_1 \Lambda$-basis 
$\varepsilon_1 Z_{K_\infty / k, S}^{0} \in \varepsilon_1 \Det_{\Lambda}^{-1} (\Delta_{K_\infty, S} (0))$ when $j=0$. 
For any negative integer $j$, we consider the following $\twist_{-j}$-semilinear isomorphism 
\[
\Twist_{0, j}^\Delta : \varepsilon_1 \Det_{\Lambda}^{-1} (\Delta_{K_\infty, S} (0)) 
\xrightarrow{\sim} \varepsilon_{1-j} \Det_{\Lambda}^{-1} (\Delta_{K_\infty, S} (j))
\]
which is defined in \S \ref{ss:twist_CM}. 
Define a $\varepsilon_{1-j} \Lambda$-basis 
$\varepsilon_{1-j} Z_{K_\infty / k, S}^{j} \in \varepsilon_{1-j} \Det_{\Lambda}^{-1} (\Delta_{K_\infty, S} (j))$ by 
\[
\varepsilon_{1-j} Z_{K_\infty / k, S}^{j} := \Twist_{0, j}^\Delta  (\varepsilon_1 Z_{K_\infty / k, S}^{0})
 \in \varepsilon_{1-j} \Det_{\Lambda}^{-1} (\Delta_{K_\infty, S} (j)). 
\]
Then, the claim (2) is trivial by definition. 
We prove the claim (1). 

By Lemma \ref{lem:ETNC^-} (1) and the definition of $\varepsilon_{1} Z_{K_\infty / k, S}^{0} \in \varepsilon_{1} \Det_{\Lambda}^{-1} (\Delta_{K_\infty, S} (0))$, 
the claim (1) is valid when $j=0$. 
Thus, we check the claim (1) when $j<0$. 
We consider the following commutative diagram: 
\[\xymatrix{
\varepsilon_{1} \Det_{\Lambda}^{-1} (\Delta_{K_\infty, S} (0)) \ar@{^{(}_->}[r]
 \ar[d]^-{\sim}_{\Twist_{0, j}^\Delta} 
&
\varepsilon_{1} Q \otimes_{\Lambda} \Det_{\Lambda}^{-1} (\Delta_{K_\infty, S} (0)) 
\ar[r]^-{\sim}
& \varepsilon_{1} Q  \ar[d]^-{\sim}_{\twist_{-j}} \\
\varepsilon_{1-j} \Det_{\Lambda}^{-1} (\Delta_{K_\infty, S} (j)) \ar@{^{(}_->}[r] 
&
\varepsilon_{1-j} Q \otimes_{\Lambda} \Det_{\Lambda}^{-1} (\Delta_{K_\infty, S} (j)) 
\ar[r]^-{\sim}
& \varepsilon_{1-j} Q 
}\]
(the commutativity can be checked as in the proof of \cite[Lemma 4.5]{Tsoi19}). 
Then, by Proposition \ref{prop:zeta=p_adic_L}, 
the bottom horizontal composite arrow sends the basis $\varepsilon_{1-j} Z_{K_\infty / k, S}^{j} $ to 
the $(1-j)$-twist of the $p$-adic $L$-function $\twist_{1-j} (\cL_{K_\infty / k , S})$. 
Since $j<0$, the complex $\C_p \otimesL_{\Z_p} \varepsilon_{1-j} \Delta_{L, S} (j)$ is acyclic by Lemma \ref{lem:complex} (1) 
and $\vartheta_{L, S}^{1-j}$ is the canonical map for any $L \in \Omega (K_\infty / k)$. 
Therefore, when $j<0$, the claim (1) follows from the interpolation property of the $p$-adic $L$-function \eqref{interpolation property}.  
This completes the proof of Theorem \ref{thm:zeta<1}. 
\end{proof}

%%%%%%%%%%%%%%%%%%%%%%%%%%%%%%%%%%%%%%%%%%%%%%%%%%%%%%%%%
\section{Functional equation of zeta elements}\label{sec:func_eq}
%%%%%%%%%%%%%%%%%%%%%%%%%%%%%%%%%%%%%%%%%%%%%%%%%%%%%%%%%

In this section, we prove Theorem \ref{thm:zeta>0}. 

We fix an odd prime $p$ and an isomorphism $\C\simeq \C_p$ and identify them. 
Let $K/k$ be a finite abelian CM extension, and we use the same notation and conventions as in §\ref{ss:main_conj}. 
In addition, we assume that $p$ is  {\it unramified} in $K / \Q$. 

%%%%%%%%%%%%%%%%%%%%%%%%%%%%%%%%%%%%%%%%%%%%%%%%%
\subsection{Definition of the basis}\label{ss:def_basis}
%%%%%%%%%%%%%%%%%%%%%%%%%%%%%%%%%%%%%%%%%%%%%%%%%

First, we state a main theorem of \cite{ADK}. 

Recall that $S$ is a finite set of places of $k$ which contains $S_p \cup S_\infty \cup S_{\ram} (K/k)$ and 
we set $S_f := S \setminus S_\infty$. 
For any integer $j$, we define an invertible $\Lambda$-module $\Xi_{K_\infty / k, S}^{\loc} (j)$ by
\[
\Xi_{K_\infty / k, S}^{\loc} (j) := \varprojlim_{L \in \Omega(K_\infty / k)} \Xi_{L/ k, S}^{\loc} (j), 
\]
where $\Xi_{L/ k, S}^{\loc} (j)$ is defined in Definition \ref{def:Xi_loc} and the projective limit is taken with respect to 
the canonical transition map in \S \ref{compatibility}. 
We put $r_k := [k : \Q]$. 

\begin{thm}. \label{thm:local_ETNC}
Recall that we assume that $p$ is unramified in $K / \Q$. 
Then, for any $j \in \Z_{\leq 0}$, 
there exist a (unique) $\varepsilon_{1-j} \Lambda$-basis 
\[
\varepsilon_{1-j} Z^{\loc, j}_{K_\infty / k, S} \in \varepsilon_{1-j} \Xi_{K_\infty / k, S}^{\loc} (j)
\]
 
such that the following statements hold. 

\begin{itemize}
\item[(1)] For any $L \in \Omega (K_\infty / k)$, 
the following composite map 
\[
\varepsilon_{1-j} \Xi_{K_\infty / k, S}^{\loc} (j) \twoheadrightarrow \varepsilon_{1-j} \Xi_{L / k, S}^{\loc} (j)
\overset{\Phi_{K/k, S}^{\loc, j}} {\hookrightarrow} \varepsilon_{1-j} \C_p [\GG_L]
\]
sends $\varepsilon_{1-j} Z^{\loc, j}_{K_\infty / k, S}$ to 
\[
(-1)^{r_k j} \times 
\begin{cases}
 \cfrac{\varepsilon_{1-j} \Theta_{L/k, S} (1-j)^\#}{\varepsilon_{1-j} \Theta_{L/k, S}(j)}
&\text{ if }j<0, 
\\
 (\prod_{v \in S_f} \delta_v (L/k))^{-1} \cdot  
\cfrac{\varepsilon_{1} \Theta_{L/k, S}(1)^\#}{\varepsilon_{1} \Theta_{L/k}(0)}
&\text{ if }j=0. 
\end{cases}
\]

\item[(2)] 
For any non-positive integers $j$ and $j'$, we have 
\[
\Twist_{j, j'}^{\loc} (\varepsilon_{1-j} Z^{\loc, j}_{K_\infty / k, S}) = \varepsilon_{1-j'} Z^{\loc, j'}_{K_\infty / k, S},  
\]
where 
$\Twist_{j, j'}^{\loc} : \varepsilon_{1-j}\Xi_{K_\infty / k, S}^{\loc} (j) 
\xrightarrow{\sim}
\varepsilon_{1-j'} \Xi_{K_\infty / k, S}^{\loc} (j')$ is defined in \eqref{loc_Twist} in \S \ref{ss:twist_def}. 
\end{itemize}
\end{thm}
\begin{proof}
The claim (1) follows from \cite[Theorem 8.3]{ADK}. 
Here we note that $\varepsilon_{1-j} \cdot c_{L / k}^{\infty} = \varepsilon_{1-j} \cdot  c^{r_k} = (-1)^{r_k (j-1)} \cdot \varepsilon_{1-j} $
($c_{L / k}^{\infty} \in \GG_L$ is defined in loc.sit.) and the map $\Phi_{K/k, S}^{\loc, j} $ is 
$(-1)^{r_k}$ times the map $\vartheta_{K/k, S}^{\loc, j} $ in \cite[Theorem 8.3]{ADK} (see Remark \ref{rem:sign}). 
Therefore, in the above equation, $(-1)^{r_k j}$ appears. 

The claim (2) follows from \cite[Proposition 8.6]{ADK}. 
\end{proof}

For any $j \in \Z_{\geq 1}$, we consider the Iwasawa cohomology complex
$\RG_{\Iw} (G_{K_\infty, S} , \Z_p (1-j))$ 
which is defined by the projective imit of $\RG (G_{L, S} , \Z_p (1-j))$ for $L \in \Omega (K_\infty / k)$. 
Then, we define the graded invertible $\Lambda$-module 
\[
\Xi_{K_\infty / k, S} (j) := \Det_{\Lambda}^{-1} (\RG_{\Iw} (G_{K_\infty, S} , \Z_p (1-j))) 
\otimes_{\Lambda} \Det_{\Lambda}^{-1} (X_{K_\infty} (-j)^+), 
\]
where we put $X_{K_\infty}(-j)^+ := \varprojlim_{L \in \Omega (K_\infty / k)} X_L (-j)^+$. 
Using the bases in Theorem \ref{thm:zeta<1} and Theorem \ref{thm:local_ETNC}, 
we define a $\varepsilon_j \Lambda$-basis of $\varepsilon_j \Xi_{K_\infty / k, S} (j)$ for any $j \in \Z_{\geq 1}$ as follows. 

Note that we have a canonical isomorphism 
\[
\Xi_{K_\infty / k, S} (j) \simeq \varprojlim_{L \in \Omega (K_\infty / k)} \Xi_{L / k, S} (j).  
\]
Similarly, there are canonical isomorphisms as above for $\Xi_{K_\infty / k, S}^{\loc} (j)$ and $\Det_{\Lambda}^{-1}(\Delta_{K_\infty , S} (j))$. 
Thus, by Lemma \ref{lem:theta_func}, 
the isomorphism $\theta_{L, S}^{1-j}$ in Definition \ref{defn:AV} induces an isomorphism
\[
\theta_{K_\infty, S}^{1-j} %:= \varprojlim_{L \in \Omega (K_\infty / k)} \theta_{L, S}^{1-j} 
: 
\varepsilon_j \Xi_{K_\infty / k, S} (j) \xrightarrow{\sim}
\varepsilon_j\Det_{\Lambda}^{-1}(\Delta_{K_\infty , S} (1-j))
\otimes_{\varepsilon_j \Lambda}
\varepsilon_j \Xi_{K_\infty / k, S}^{\loc} (1-j)
\]
for any $j \in \Z_{\geq 1}$. 
Using this,  
we define a $\varepsilon_j \Lambda$-basis of $\varepsilon_j \Xi_{K_\infty / k, S} (j)$ as follows. 

\begin{defn}\label{def:even zeta element}
For any $j \in \Z_{\geq 1}$, 
we define a $\varepsilon_j \Lambda$-basis 
$\varepsilon_j Z_{K_\infty/k, S}^{j} \in \varepsilon_j \Xi_{K_\infty / k, S} (j)$ by 
\[
\varepsilon_j Z_{K_\infty /k, S}^{j} := 
%(-1)^{r_k} \cdot 
(\theta_{K_\infty, S}^{1-j})^{-1} \parenth{\varepsilon_j Z_{K_\infty/k, S}^{1-j} \otimes 
\varepsilon_j Z_{K_\infty / k, S}^{\loc, 1-j}}, 
\]
where $\varepsilon_j Z_{K_\infty/k, S}^{1-j}$ is the basis in Theorem \ref{thm:zeta<1} and   
$\varepsilon_j Z_{K_\infty / k, S}^{\loc, 1-j}$ is the one in Theorem \ref{thm:local_ETNC}. 
\end{defn}

%%%%%%%%%%%%%%%%%%%%%%%%%%%%%%%%%%%%%%%%%%%%%%%
\subsection{Proof of Theorem \ref{thm:zeta>0}}
%%%%%%%%%%%%%%%%%%%%%%%%%%%%%%%%%%%%%%%%%%%%%%%
In this subsection, we prove Theorem \ref{thm:zeta>0}. 
\begin{proof}[Proof of Theorem \ref{thm:zeta>0}]
By Theorem \ref{thm: comm of period reg}, 
the following diagram is commutative:
%up to sign $(-1)^{r_k}$

\[
\xymatrix{ 
\varepsilon_j \Xi_{K_\infty / k, S} (j) 
\ar[r]_-{\sim}^-{\theta_{K_\infty , S}^{1-j}}
\ar@{->>}[d]
&
\varepsilon_j \Det_{\Lambda}^{-1}(\Delta_{K_\infty , S} (1-j))
\otimes_{\varepsilon_j \Lambda}
\varepsilon_j \Xi_{K_\infty / k, S}^{\loc} (1-j)
\ar@{->>}[d]
\\
\varepsilon_j\Xi_{L/ k, S} (j) 
\ar[r]_-{\sim}^-{\theta_{L, S}^{1-j}}
\ar[d]^-{\vartheta_{L/k, S}^j}
&
\varepsilon_j\Det_{\Z_p[\GG_L]}^{-1}(\Delta_{L , S} (1-j))
\otimes_{\varepsilon_j \Z_p[\GG_L]}
\varepsilon_j\Xi_{L / k, S}^{\loc} (1-j)
\ar[d]^-{\vartheta_{L/k, S}^{1-j} \otimes \Phi_{L / k, S}^{\loc, 1-j}}
\\
\varepsilon_j\C_p [\GG_L]
&
\ar[l]_-{ab \mapsfrom a \otimes b}^-{\sim}
\varepsilon_j\C_p [\GG_L]
\otimes_{\varepsilon_j\C_p [\GG_L]}
\varepsilon_j\C_p [\GG_L]
}
\]
for any $L \in \Omega (K_\infty / k)$. 
By Theorem \ref{thm:zeta<1} (1) and Theorem \ref{thm:local_ETNC} (1),
the right vertical composite map sends 
$\varepsilon_jZ_{K_\infty/k, S}^{1-j} \otimes 
\varepsilon_jZ_{K_\infty / k}^{\loc, 1-j}$ to 
\[
\begin{cases}\displaystyle
\varepsilon_j \Theta_{L/k , S} (1-j) \otimes
(-1)^{r_k (j-1)} 
\cfrac{\varepsilon_j  \Theta_{L/k , S} (j)^\#}{\varepsilon_j  \Theta_{L/k , S} (1-j)}
& \text{if }  j>1,    
\\ \displaystyle
\Big{(} \prod_{v \in S_f } \delta_v(L/k) \Big{)}
\varepsilon_1\Theta_{L/k} (0) \otimes 
\Big{(} \prod_{v \in S_f } (\delta_v(L/k) \Big{)^{-1}} \cdot
\cfrac{\varepsilon_1 \Theta_{L/k , S} (1)^\#}{\varepsilon_1 \Theta_{L/k} (0)}
& \text{if }  j=1.  
\end{cases}
\]
This implies Theorem \ref{thm:zeta>0} (1). 

 On the other hand, by Definition \ref{defn:AV} and the diagram \eqref{triangle_twist}, 
 we have the following commutative diagram: 
 \[
\xymatrix{ 
\varepsilon_j\Xi_{K_\infty / k, S} (j) 
\ar[r]_-{\sim}^-{\theta_{K_\infty , S}^{1-j}}
\ar[d]^-{\sim}_-{\Twist_{1-j, 1-j'}^\Xi}
&
\varepsilon_j\Det_{\Lambda}^{-1}(\Delta_{K_\infty , S} (1-j))
\otimes_{\varepsilon_j \Lambda}
\varepsilon_j\Xi_{K_\infty / k, S}^{\loc} (1-j)
\ar[d]^-{\sim}_-{\Twist_{1-j, 1-j'}^\Delta \otimes \Twist_{1-j, 1-j'}^{\loc}}
\\
\varepsilon_{j'}\Xi_{K_\infty / k, S} (j') 
\ar[r]_-{\sim}^-{\theta_{K_\infty , S}^{1-j'}}
&
\varepsilon_{j'}\Det_{\Lambda}^{-1}(\Delta_{K_\infty , S} (1-j'))
\otimes_{\varepsilon_{j'} \Lambda}
\varepsilon_{j'} \Xi_{K_\infty / k, S}^{\loc} (1-j'). 
}
\]
The claim (2) in Theorem \ref{thm:zeta>0} follows from this diagram, Theorem \ref{thm:zeta<1} (2), and Theorem \ref{thm:local_ETNC} (2). 
This completes the proof of Theorem \ref{thm:zeta>0}. 
\end{proof}

%%%%%%%%%%%%%%%%%%%%%%%%%%%%%%%%%%%%%%%%%%%%5
\section{Generalized Stark element}\label{section:eulersystem}
%%%%%%%%%%%%%%%%%%%%%%%%%%%%%%%%%%%%%%%%%%%%%%%%%%%%%%%%%%%%
We keep the notation in the previous section. 
In this section, %as an application of Theorem \ref{thm:zeta>0}, 
for any $j \in \Z_{\geq 1}$ and $L \in \Omega (K_\infty / k)$, 
we define an element 
$\eta_{L, S}^j $ which belongs to the exterior power bidual of the cohomology group
 $H^1 (G_{L , S} , \Z_p(1-j))$ 
 using Theorem \ref{thm:zeta>0}. 

Recall that $p$ is an odd prime and $K/k$ is a finite abelian CM-extension such that $p$ is unramified in $K / \Q$. 
We set $K_\infty := \cup_{n \geq 0} K (\mu_{p^n})$. 
We fix a finite set $S$ of places of $k$ which contains $S_p \cup S_{\ram} (K/k) \cup S_\infty$. 

Throughout this section, we fix $L \in \Omega (K_\infty / k)$. 
Namely, $L$ is an intermediate field of $K_\infty / k$ such that $L$ is a CM-field and $[L : k] <\infty$. 
Put $r_k := [k : \Q]$ and $\GG_L := \Gal (L/k)$. 

We fix a numbering of embeddings
$\Sigma_{k} := \{ \iota : k \hookrightarrow \R \} = \{\iota_1 , \dots , \iota_{r_k} \}$. 
We also fix a lift $\iota_{L, i} : L \hookrightarrow \C$ of $\iota_i$ for each $1 \leq i \leq r_k$. 
Then, we obtain an ordered $\Z_p[\GG_L]$-basis $\{ e_{\iota_{L, 1}}^j, \dots , e_{\iota_{L, r_k}}^j \}$ of 
$X_L (j)$ for any integer $j$, 
where $X_L (j)$ and $e_{\iota_{L, i}}^j$ is defined in \S \ref{ss:notation}. 

%%%%%%%%%%%%%%%%%%%%%%%%%%%%%%%%%%%%%%%%%%%%5
\subsection{Definition of $\eta_{L, S}^j $}\label{ss:def_eta}
%%%%%%%%%%%%%%%%%%%%%%%%%%%%%%%%%%%%%%%%%%%%%%%%%%%%%%%%%%%%
For any commutative ring $R$, any $R$-module $M$, and any positive integer $r$, 
we define the $r$-th exterior power bidual of $M$ by
\[
\bigcap_R^r M := 
\Hom_R (\bigwedge_R^r \Hom_R (M, R) , R). 
\]
The goal of this subsection is to define an element 
\[
\eta_{L, S}^j \in \varepsilon_j\bigcap_{\Z_p[\GG_{L}]}^{r_k}  H^1 (G_{L , S(L)} , \Z_p(1-j))
\] 
for any $j \in \Z_{\geq 1}$. 

First, we give an elementary lemma. 
\begin{lem}\label{lem:H^1torsion-free}
For any $j \in \Z_{\geq 1}$, 
$\varepsilon_j H^1 (G_{L , S} , \Z_p(1-j))$ is $\Z_p$-torsion free. 
\end{lem}
\begin{proof}
Consider the following short exact sequence 
\[
0 \to \Z_p (1-j) \xrightarrow{\times p} \Z_p (1-j) \to \Z / p \Z (1-j) \to 0.   
\]
Taking the Galois cohomology, we get a surjective homomorphism 
\[
\varepsilon_jH^0 (G_{L, S} , \Z / p \Z (1-j)) \twoheadrightarrow  \varepsilon_jH^1 (G_{L, S} , \Z_p (1-j)) [p], 
\]
where we write $(-)[p]$ for the $p$-torsion part of $(-)$. 
Clearly, we know that the module $\varepsilon_jH^0 (G_{L, S} , \Z / p \Z (1-j))$ vanishes. 
This implies the claim. 
\end{proof}

We consider the following composite map: 
\begin{align*}
\Pi_{L}^j : 
%\varepsilon_j\Xi_{K_\infty / k , S} (j) 
%\twoheadrightarrow 
\varepsilon_j\Xi_{L / k , S} (j)  
&\hookrightarrow
\Q_p \otimes_{\Z_p} \varepsilon_j\Xi_{L / k , S} (j) 
\\
&\xrightarrow{\sim} 
\varepsilon_j\Det_{\Q_p [\GG_L]} (H^1 (G_{L, S} , \Q_p (1-j))) 
\otimes_{\varepsilon_j\Z_p[\GG_L]} \varepsilon_j\Det_{\Z_p[\GG_L]}^{-1} (X_K (-j)^+) 
\\
&\xrightarrow{\sim} 
\varepsilon_j\Det_{\Q_p [\GG_L]} (H^1 (G_{L, S} , \Q_p (1-j))) 
= \varepsilon_j \bigwedge_{\Q_p [\GG_L]}^{r_k} (H^1 (G_{L, S} , \Q_p (1-j))), 
\end{align*}
%where the first surjection is the canonical transition map, 
where the first isomorphism is induced by Lemma \ref{lem:complex} (3) 
and the second isomorphism is defined by 
\[
x \otimes \varepsilon_j(\wedge_{i=1}^{r_k} e_{\iota_{L, i}}^{-j})^\ast \mapsto x 
\]
($(-)^\ast$ means the dual basis). 

On the other hand, for any $\Z_p[\GG_L]$-module $M$ and $r \in \Z_{\geq 1}$, 
the following natural map 
\begin{equation}\label{wedge}
\bigwedge_{\Z_p[\GG_{L}]}^{r} M 
\to \bigcap_{\Z_p[\GG_{L}]}^{r} M ; \;
x \mapsto (f \mapsto f(x))
\end{equation}
induces an isomorphism 
$
\Q_p \otimes_{\Z_p} \bigwedge_{\Z_p[\GG_{L}]}^{r} M 
\simeq \Q_p \otimes_{\Z_p} \bigcap_{\Z_p[\GG_{L}]}^{r} M 
(\supset \bigcap_{\Z_p[\GG_{L}]}^{r} M). $
Hence, we may regard 
$\bigcap_{\Z_p[\GG_{L}]}^{r} M 
\subset \Q_p \otimes_{\Z_p}  \bigwedge_{\Z_p[\GG_{L}]}^{r} M  
= \bigwedge_{\Q_p[\GG_{L}]}^{r} (\Q_p \otimes_{\Z_p} M )$. 
Note that for any finitely generated $\Z_p[\GG_L]$-module $M$, 
\eqref{wedge} induces an isomorphism 
\[
\bigcap_{\Z_p[\GG_L]}^{r} M \simeq 
\{ 
x \in \bigwedge_{\Z_p[\GG_L]}^{r} \Q_p \otimes_{\Z_p} M \mid 
f (x) \in \Z_p[\GG_L] \text{ for any } f \in \bigwedge_{\Z_p[\GG_L]}^{r}
\Hom_{\Z_p[\GG_L]} (M , \Z_p [\GG_L]) 
\}
\]
(see \cite[Proposition A.8]{BS19}). 
%and 
%$\eta_L^j \in \varepsilon_j \bigcap_{\Z_p[\GG_L]}^{r_k} H^1 (G_{L , S} , \Z_p(1-j)) $, 
Therefore, for any element $a \in \varepsilon_j \bigcap_{\Z_p[\GG_L]}^{r} M $, 
we can define an ideal ${\rm im } (a) \subset  \varepsilon_j \Z_p [\GG_L]$ by 
\[
{\rm im } (a) := 
\Big{\langle} f (a) \mid f \in \varepsilon_j \bigwedge_{\Z_p[\GG_L]}^{r_k} 
\Hom_{\Z_p[\GG_L]}( M , \Z_p [\GG_L] ) 
\Big{\rangle}_{\varepsilon_j\Z_p [\GG_L]}. 
\]
 
\begin{prop}\label{prop:image_Pi}
For any $j \in \Z_{\geq 1}$, the following statements hold. 
\begin{itemize}
\item[(i)]
We have 
$\Imag \Pi_L^j \subset \varepsilon_j\bigcap_{\Z_p[\GG_{L}]}^{r_k}  H^1 (G_{L , S} , \Z_p(1-j))$. 
\item[(ii)] For any $\varepsilon_j \Z_p[\GG_L]$-basis $z \in \varepsilon_j \Xi_{L/k , S} (j)$, we have 
\[
{\rm im } ( \Pi_{L}^j (z)) = \varepsilon_j\Fitt_{\Z_p[\GG_L]}^0 (H^2 (G_{L, S} , \Z_p (1-j))). 
\]
\end{itemize}
\end{prop}
\begin{proof}
We consider the complex $C(j) := \varepsilon_j \RG  (G_{L , S} , \Z_p(1-j))[1] \oplus X_K (-j)^+ [-1]$. 
We claim that the following hold for any $j \in \Z_{\geq 1}$:   
\begin{itemize}
\item[(1)] $C(j)$ is a perfect complex of $\varepsilon_j\Z_p [\GG_L]$-modules. 
\item[(2)] The Euler characteristic of $\Q_p \otimesL_{\Z_p}C(j)$ is zero. 
\item[(3)] $C(j)$ is acyclic outside degrees zero and one. 
\item[(4)] $H^0(C(j))$ is $\Z_p$-free. 
\end{itemize}
The claim (1) is well known. 
Claims (2) and (3) follow from Lemma \ref{lem:complex} (3) and the isomorphism $\lambda_{L}^j$ 
in Definition \ref{def:period_reg>0}. 
The claim (4) follows from Lemma \ref{lem:H^1torsion-free}. 

Therefore, the complex $C(j)$ satisfies the conditions \cite[Definition 2.20]{BS19} as $\mathcal{R} =\varepsilon_j \Z_p[\GG_L]$. 
Thus, Proposition \ref{prop:image_Pi} (i) follows from \cite[Proposition A.11 (ii)]{BS19}. 

Moreover, by \cite[Proposition A.11 (i)]{BS19}, 
we may assume that $C(j) = [P' \to P]$, where $P'$ and $P$ is a free $\varepsilon_j \Z_p[\GG_L]$-module of rank $d \geq r_k$ 
and $P'$ is placed in degree zero. 
Namely, we have the following exact sequence 
\[
0 \to H^0 (C (j)) \xrightarrow{f} P' \to P \to H^1 (C(j)) \to 0. 
\]
Since the cokernel of $f$ is $\Z_p$-free, 
the dual map of $f$ 
 \[
 \Hom_{\varepsilon_j\Z_p[\GG_L]}(P' , \varepsilon_j\Z_p[\GG_L]) \to
 \Hom_{\varepsilon_j\Z_p[\GG_L]} (H^0 (C(j)) , \varepsilon_j\Z_p[\GG_L])\]
 is surjective. 
 Thus, we can regard as  
$\varepsilon_j\bigcap_{\Z_p[\GG_{L}]}^{r_k}  H^1 (G_{L , S} , \Z_p(1-j)) =
\varepsilon_j\bigcap_{\Z_p[\GG_{L}]}^{r_k}  H^0 (C(j)) \subset \varepsilon_j\bigcap_{\Z_p[\GG_{L}]}^{r_k} P'$. 
Then, by \cite[Lemma A.7 (iii)]{BS19}, we have 
\begin{align*}
&\Fitt_{\varepsilon_j \Z_p[\GG_L]}^{r_k} (H^1 (C(j))) 
= \varepsilon_j \Fitt_{\Z_p[\GG_L]}^{0} (H^2 (G_{L, S} , \Z_p (1-j))) 
\\
&= \Big{\langle} \phi (\Pi_L^j (z)) \; | \; \phi \in  \bigwedge_{\varepsilon_j\Z_p[\GG_L]}^{r_k} 
\Hom_{\varepsilon_j\Z_p[\GG_L]}( P' , \varepsilon_j\Z_p [\GG_L] )  \Big{\rangle}_{\varepsilon_j\Z_p[\GG_L]}  
\\
&= \Big{\langle} \phi (\Pi_L^j (z)) \; | \; \phi \in  \bigwedge_{\varepsilon_j\Z_p[\GG_L]}^{r_k} 
\Hom_{\varepsilon_j\Z_p[\GG_L]}( H^1 (G_{L, S} , \Z_p (1-j)) , \varepsilon_j\Z_p [\GG_L] )  \Big{\rangle}_{\varepsilon_j\Z_p[\GG_L]}  
= :{\rm im } (\Pi_L^j (z)), 
\end{align*}
where the third equality follows from the surjectivity of the dual map of $f$. 
This completes the proof of Proposition \ref{prop:image_Pi}. 
\end{proof}

Now we define an element $\eta_{L, S}^j$ as follows. 

\begin{defn}\label{def:eta}
For any $j \in \Z_{\geq 1}$, we define 
\[
\eta_{L, S}^j  \in \varepsilon_j\bigcap_{\Z_p[\GG_{L}]}^{r_k}  H^1 (G_{L , S} , \Z_p(1-j)), 
\]
as the image of the basis $\varepsilon_jZ_{K_\infty / k , S}^j \in \varepsilon_j\Xi_{K_\infty / k , S} (j) $ 
in Theorem \ref{thm:zeta>0} under the following composite map 
\[
\varepsilon_j\Xi_{K_\infty / k , S} (j) \twoheadrightarrow  \varepsilon_j\Xi_{L / k , S} (j) \xrightarrow{\Pi_L^j} 
 \varepsilon_j\bigcap_{\Z_p[\GG_{L}]}^{r_k}  H^1 (G_{L , S} , \Z_p(1-j)), 
\]
where the first arrow is the natural map. 
\end{defn}
We note that the definition of $\eta_{L, S}^j$ depends on the choice of the ordered 
$\Z_p[\GG_L]$-basis $\{ e_{\iota_{L, 1}}^j, \dots , e_{\iota_{L, r_k}}^j \}$ of 
$X_L (j)$. 

%%%%%%%%%%%%%%%%%%%%%%%%%%%%%%%%%%%%%%%%%%%%%%%%%%%%%%%%%%%%
\subsection{Proof of Theorem \ref{thm:ES}}\label{ss:main2_proof}
%%%%%%%%%%%%%%%%%%%%%%%%%%%%%%%%%%%%%%%%%%%%%%%%%%%%%%%%%%%%
In this subsection, we give a precise statement of Theorem \ref{thm:ES} and prove this. 

For any $j \in \Z_{\geq 1}$, 
we consider the following $\varepsilon_j \C_p[\GG_L]$-isomorphism 
\[
\wedge \lambda_{L}^j :
\C_p \otimes_{\Q_p} 
 \varepsilon_j \bigwedge_{\Q_p[\GG_L]}^{r_k} H^1 (G_{L, S}, \Q_p (1-j) ) 
\xrightarrow{\sim} \C_p \otimes_{\Z_p}  \varepsilon_j \bigwedge_{\Z_p [\GG_L]}^{r_k} X_L (-j) 
\]
which is induced by $\lambda_{L}^j$ in Definition \ref{def:period_reg>0}. 
The following result is a partial answer of a conjecture of Burns--Kurihara--Sano (see \cite[Definition 2.9 and Conjecture 3.6]{BKS20}). 
Namely, $\eta_{L, S}^j$ coincides with a ``generalized Stark element" defined in \cite{BKS20}. 

\begin{thm}\label{thm:Euler sys and Fitt}
For any $j \in \Z_{\geq 1}$, the following statements hold. 
\begin{itemize}
\item[(i)]
We have 
\[
\varepsilon_j \Fitt_{\Z_p[\GG_L]}^0 (H^2 (G_{L , S} , \Z_p(1-j)) = {\rm im} (\eta_{L, S}^j). 
\]
\item[(ii)]
We have 
\[
\wedge \lambda_{L}^j  (\eta_{L, S}^j) = 
(-1)^{r_k (j-1)}\cdot \varepsilon_j \Theta_{L/k , S} (j)^\# \cdot  \wedge_{i=1}^{r_k} e_{\iota_{L, i}}^{-j}
\]
in $ \C_p \otimes_{\Z_p}  \varepsilon_j \bigwedge_{\Z_p [\GG_L]}^{r_k} X_L (-j) $. 
\end{itemize}
\end{thm}
\begin{proof}
The claim (i) immediately follows from Proposition \ref{prop:image_Pi} and Definition \ref{def:eta}. 
Therefore, we prove the claim (ii). 

We consider the following composite map 
\begin{align*}
\varepsilon_j\Xi_{K_\infty / k , S} (j)  
&\twoheadrightarrow
\varepsilon_j\Xi_{L / k , S} (j)  
\hookrightarrow
\Q_p \otimes_{\Z_p} \varepsilon_j\Xi_{L / k , S} (j) 
\\&\xrightarrow{\sim} 
\varepsilon_j\Det_{\Q_p [\GG_L]} (H^1 (G_{L, S} , \Q_p (1-j))) 
\otimes_{\varepsilon_j\Z_p[\GG_L]} \varepsilon_j\Det_{\Z_p[\GG_L]}^{-1} (X_K (-j)^+)
\\
&= 
\varepsilon_j \bigwedge_{\Q_p [\GG_L]}^{r_k} (H^1 (G_{L, S} , \Q_p (1-j))) 
\otimes_{\varepsilon_j\Z_p[\GG_L]} \varepsilon_j\Det_{\Z_p[\GG_L]}^{-1} (X_K (-j)^+), 
\end{align*}
where the isomorphism is induced by Lemma \ref{lem:complex} (3). 
By Definition \ref{def:eta}, we know that the image of the basis 
$\varepsilon_jZ_{K_\infty / k , S}^j \in \varepsilon_j\Xi_{K_\infty / k , S} (j) $ 
under the above composite map coincides with 
$\eta_{L, S}^j \otimes \varepsilon_j (\wedge_{i=1}^{r_k} e_{\iota_{L, i}}^{-j})^\ast$. 
Therefore, by Theorem \ref{thm:zeta>0} (1) and the definition of $\vartheta_{L/ k, S}^{j}$ in Definition \ref{def:period_reg>0}, 
the isomorphism 
\begin{align*}
\C_p \otimes_{\Z_p} \varepsilon_j \bigwedge_{\Q_p [\GG_L]}^{r_k} (H^1 (G_{L, S} , \Q_p (1-j))) 
\otimes_{\varepsilon_j\Z_p[\GG_L]} \varepsilon_j\Det_{\Z_p[\GG_L]}^{-1} (X_K (-j)^+)
\\ 
\overset{\wedge \lambda_{L}^j}{\xrightarrow{\sim}}
\C_p \otimes_{\Z_p} \varepsilon_j\Det_{\Z_p[\GG_L]} (X_K (-j)^+)
\otimes_{\varepsilon_j\Z_p[\GG_L]} \varepsilon_j\Det_{\Z_p[\GG_L]}^{-1} (X_K (-j)^+)
\overset{\ev}{\xrightarrow{\sim}}
\varepsilon_j \C_p [\GG_L]
\end{align*}
sends the element $\eta_{L, S}^j \otimes \varepsilon_j (\wedge_{i=1}^{r_k} e_{\iota_{L, i}}^{-j})^\ast$ to 
$(-1)^{r_k (j-1)}\cdot \varepsilon_j\Theta_{L/k , S} (j)^\#$. 
The claim (ii) follows from this. 
\end{proof}

Next, we give a congruence relation between 
$\eta_{L, S}^j$ and $\eta_{L, S}^{j'}$ for different integers $j \text{ and } j'$. 
This congruence is also conjectured by Burns--Kurihara--Sano 
for generalized Stark elements (cf. \cite[Conjecture  2.11]{Tsoi19}). 

We fix a positive integer $n$ such that $\mu_{p^n} \subset L$. 
We write $\bar{\chi}_{\cyc, n}$ for the associated cyclotomic character 
$\Gal (L / k) \to \Aut (\mu_{p^n}) \simeq (\Z / p^n )^\times$. 
For each integer $m$, 
we define 
$\overline{\twist_{L, n}^m}$ to be the ring automorphism of $\Z / p^n  [\GG_L]$ which sends 
$\sigma$ to $\bar{\chi}_{\cyc, n} (\sigma)^m \sigma$. 
%Clearly, this induces a ring isomorphism 
%$\overline{\twist}_{K, n}^m : \varepsilon_j\Z / p^n  [\GG_K] \xrightarrow{\sim} e_{j+m}^+\Z / p^n  [\GG_K]$ 
%for any $j \in \Z$. 

Recall that we fixed an embedding $\iota_{L, i} : L \hookrightarrow \C$ for each $1 \leq i \leq r_k $ at 
 the beginning of \S \ref{section:eulersystem}. 
Then, for each $1 \leq i \leq r_k $, we set 
\[
\xi_i := \iota_{L, i}^{-1}  \parenth{\exp (\frac{2 \pi \sqrt{-1}}{p^n})} \in \mu_{p^n} (L)
= H^0 (G_{L, S}, \Z / p^n  (1)). 
\]
Using this, for any $1 \leq i \leq r_k$ and any integers $j$ and $j'$, 
we consider the map 
\begin{align*}
c_i : \Hom_{\Z / p^n  [\GG_L]} (H^1 (G_{L , S} , \Z / p^n (1-j')) ,  \Z/ p^n [\GG_L]) \\
\to \Hom_{\Z / p^n [\GG_L]} (H^1 (G_{L , S} , \Z / p^n(1-j)) , \Z / p^n  [\GG_L])
\end{align*}
induced by taking cup product with $\xi_i^{\otimes j - j'}$. 
Then,  we define a $\overline{\twist_{L, n}^{j' - j}}$-semilinear isomorphism  
\[
\overline{\Twist}_{1-j , 1-j' , L, n} :  \bigcap_{\Z / p^n [\GG_{L}]}^{r_k} H^1 (G_{L , S} , \Z / p^n(1-j) )
\xrightarrow{\sim}  \bigcap_{\Z / p^n  [\GG_{L}]}^{r_k}H^1 (G_{L , S} , \Z / p^n (1-j'))
\]
by the $\Z / p^n [\GG_L]$-linear dual of the following map 
\begin{align*}
\bigwedge_{\Z / p^n [\GG_L]}^{r_k} 
\Hom_{\Z / p^n  [\GG_L]} (H^1 (G_{L , S} , \Z / p^n (1-j')) ,  \Z/ p^n [\GG_L]) \\
\to \bigwedge_{\Z / p^n [\GG_L]}^{r_k} 
\Hom_{\Z / p^n [\GG_L]} (H^1 (G_{L , S} , \Z / p^n(1-j)) , \Z / p^n  [\GG_L]) 
\end{align*}
that sends $\wedge_{i=1}^{r_k} a_i$ to $\wedge_{i=1}^{r_k} c_i (a_i)$. 
Since $\overline{\twist_{L, n}^{j' - j}} (\varepsilon_j) = \varepsilon_{j'}$, this induces 
\[
\overline{\Twist}_{1-j , 1-j' , L, n} :  \varepsilon_j \bigcap_{\Z / p^n [\GG_{L}]}^{r_k} H^1 (G_{L , S} , \Z / p^n(1-j) )
\xrightarrow{\sim}  \varepsilon_{j'}\bigcap_{\Z / p^n  [\GG_{L}]}^{r_k}H^1 (G_{L , S} , \Z / p^n (1-j')). 
\]

We can now state a congruence relation between $\eta_{L, S}^j$ and $\eta_{L, S}^{j'}$. 
We write $\ol{\eta_{L, S}^j }$ 
for the image of $\eta_{L, S}^j$ under the following reduction map modulo $p^n$ (cf. \cite[Lemma 2.4]{Tsoi19}): 
\[
{\rm red}_{p^n} : \varepsilon_j \bigcap_{\Z_p [\GG_{L}]}^{r_k} H^1 (G_{L , S} , \Z _p(1-j) ) \to 
\varepsilon_j \bigcap_{\Z / p^n [\GG_{L}]}^{r_k} H^1 (G_{L , S} , \Z / p^n(1-j) ) . 
\]

\begin{thm}\label{thm:generalised kummer congruence}
For any positive integers $j$ and $j'$ and any $n \geq 0$ such that $\mu_{p^n} \subset L$, we have 
\[
\overline{\Twist}_{1-j , 1-j' , L, n} (\ol{\eta_{L, S}^j}) = \ol{\eta_{L, S}^{j'}} 
\]
in $\bigcap_{\Z / p^n  [\GG_{L}]}^{r_k}H^1 (G_{L , S} , \Z / p^n (1-j'))$. 
\end{thm}
\begin{proof}
%For any even integer $m$, 
%we write $\twist_{m, \cK_\infty^+}$ for the ring automorphism of $\Lambda_{\cK_\infty^+}$ induced by 
%the $m$-th cyclotomic character. 
For any integers $j$ and $j'$, 
we consider the isomorphism  
\[
\Twist_{1-j,1-j'}^\Xi : \varepsilon_j \Xi_{K_\infty / k, S} (j) \to \varepsilon_{j'}\Xi_{K_\infty / k, S} (j') 
\] 
which is defined in \S \ref{ss:twist_CM}. 
Then, by \cite[Proposition 4.8]{Tsoi19}, we have the following commutative diagram  
\[
\xymatrix@C=40pt{
\varepsilon_j\Xi_{K_\infty / k, S} (j) \ar[d]_-{\Twist_{1-j,1-j'}^\Xi } \ar@{->>}[r]
&\varepsilon_j \Xi_{L/ k, S} (j) \ar[r]^-{{\rm red}_{p^n} \circ \Pi_L^{j}}
&
\varepsilon_j \bigcap_{\Z / p^n  [\GG_{L}]}^{r_k}H^1 (G_{L , S} , \Z / p^n (1-j)) 
\ar[d]^-{\overline{\Twist}_{1-j , 1-j' , L, n}}
\\ 
\varepsilon_j \Xi_{K_\infty / k, S} (j') \ar@{->>}[r]
& \varepsilon_j \Xi_{L / k, S} (j')  \ar[r]^-{{\rm red}_{p^n} \circ \Pi_L^{j'}}
%&\bigcap_{\Z_p  [\GG_{K}]}^{r_k}H^1 (G_{K , S(K)} , \Z_p (1-j')) \ar[r]
&
\varepsilon_j\bigcap_{\Z / p^n  [\GG_{L}]}^{r_k}H^1 (G_{L , S} , \Z / p^n (1-j')), 
}
\] 
Therefore, the claim follows from Theorem \ref{thm:zeta>0} (2) and 
Definition \ref{def:eta}. 
\end{proof}

%%%%%%%%%%%%%%%%%%%%%%%%%%%%%%%%%%%%%%%%%%%%%%%%%%%%%%%%%
\section{The case when $k=\Q$}\label{k=Q}
%%%%%%%%%%%%%%%%%%%%%%%%%%%%%%%%%%%%%%%%%%%%%%%%%%%%%%%%%
In this section, 
we consider the case when the base field $k$ is the rational field $\Q$. 
In this case, the basis in Theorem \ref{thm:zeta>0} coincides with
a zeta element which is defined by Kato using cyclotomic units in \cite[Chapter I, 3.3.5]{Kato91}. 
We explain this here.

%%%%%%%%%%%%%%%%%%%%%%%%%%%%%%%%%%%%%%%%%%%%%%%%%%%%%%%%%
\subsection{Cyclotomic units}
%%%%%%%%%%%%%%%%%%%%%%%%%%%%%%%%%%%%%%%%%%%%%%%%%%%%%%%%%
We fix an odd prime number $p$. 
Let $N'$ be a positive integer such that $N' \not\equiv 2 \pmod 4$ and $N'$ is prime to $p$. 
We consider the $N'$-th cyclotomic field $K:= \Q (\mu_{N'})$ and 
set $K_n := K (\mu_{p^n})$ for any $n \in \Z_{\geq 0}$. 
 
For simplicity, we take $S$ to be the minimal set of bad primes of $ K / \Q$. 
Namely, let $S$ be the set of places of $\Q$ which divides $N'p\infty$. 

For any $n \geq 1$, 
we consider the cyclotomic unit 
\[
c_{N'p^n} := 
(1-\alpha_{N'p^n})(1-\alpha_{N'p^n}^{-1}) \in \R^\times, 
\]
where $\alpha_{N'p^n} := \exp (\frac{2 \pi \sqrt{-1}}{N'p^n}) \in \C$. 
Note that we can regard 
\[
\iota_{K_n}^{-1} (c_{N'p^n}) \in H^1 (G_{K_{n} , S} , \Z_p (1)) 
\]
for any embedding $\iota_{K_n} : K_{n} \hookrightarrow \C$ by Kummer theory. 
Then, we define an element 
\[
z_{N'p^n}^{\cyc} \in 
%\varprojlim_n \Q_p \otimes_{\Z_p} \Xi_{K_{Np^n} / \Q} (0):= \varprojlim_n 
H^1 (G_{K_{n} , S} , \Z_p (1)) \otimes_{\Z_p [\Gal (K_n / \Q)]} 
X_{K_{n}} (0)^{-1}
\]
by
\[
z_{N'p^n}^{\cyc} := \frac{1}{2} \cdot \iota_{K_n}^{-1} (c_{N'p^n}) \otimes e_{\iota_{K_n}}^{0, \ast}, 
\]
where $\iota_{K_n} : K_{n} \hookrightarrow \C$ is an embedding and 
\[
e_{\iota_{K_n}}^{0, \ast} \in X_{K_{n}} (0)^{-1} := \Hom_{\Z_p[\Gal (K_n / \Q)]} 
(\bigoplus_{\iota : K_{n} \hookrightarrow \C } \Z_p , \Z_p[\Gal (K_n / \Q)] )
\] 
is the dual basis of $e_{\iota_{K_n}}^{0}$ defined in \S \ref{ss:notation}. 
Note that the definition of $z_{N'p^n}^{\cyc}$ does not depend on the choice of an embedding $\iota_{K_n}$. 

Put $K_\infty := \cup_{n \geq 0} K_n$, $\Lambda := \Z_p [[\Gal (K_{\infty} / \Q)]]$ and 
$\varepsilon_j := \frac{1 + (-1)^j c}{2} \in \Lambda$ for any $j \in \Z$ as in the previous section, 
where $c \in \Gal (K_{\infty} / \Q)$ is the complex conjugation. 
For any intermediate field $\Q \subset L \subset K_{\infty}$ such that 
$[L : \Q] <\infty$ and $L$ is a CM-field, we also write $\varepsilon_j$ for the image 
of $\varepsilon_j$ under the natural restriction map $\Lambda \twoheadrightarrow \Z_p[\Gal (L / \Q)]$ but this causes no confusion. 
Then, we define an element 
\[
\varepsilon_0Z_{N'p^\infty}^{\cyc} \in 
%e^+_0 \Q_p [[\Gal (K_{Np^\infty} / \Q)]] \otimes_{\Lambda_{Np^\infty}} \Xi_{K_{Np^\infty} / \Q , S_N} (0)= 
\varepsilon_0 \parenth{ \varprojlim_n \Q_p \otimes_{\Z_p} \Xi_{K_{n} / \Q, S} (0) }
\]
by the projective limit of 
\[
\varepsilon_0 z_{N'p^n}^{\cyc} \in 
\Q_p \otimes_{\Z_p}
\varepsilon_0 \parenth{
H^1 (G_{K_{n} , S} , \Z_p (1)) \otimes_{\Z_p [\Gal (K_n / \Q)]} 
X_{K_{n}} (0)^{-1}
}
= 
\varepsilon_0 \parenth{ \Q_p \otimes_{\Z_p} \Xi_{K_{n} / \Q, S} (0) }. 
\]
We note that 
the twist map 
$ \Twist_{1, 1-j}^{\Xi} : \varepsilon_0\Xi_{K_{\infty} / \Q, S} (0) \xrightarrow{\sim}  \varepsilon_j\Xi_{K_{\infty} / \Q, S} (j)$
in \S \ref{ss:twist_CM} can be extended
\[
 \Twist_{1, 1-j}^{\Xi} : 
\varepsilon_0\parenth{ \varprojlim_n \Q_p \otimes_{\Z_p} \Xi_{K_{n} / \Q, S} (0) } 
\xrightarrow{\sim} 
\varepsilon_j\parenth{ \varprojlim_n \Q_p \otimes_{\Z_p} \Xi_{K_{n} / \Q, S} (j) }
\]
naturally. 

The following result is proved by Kato in \cite[Chapter III, Theorem 1.2.6]{Kato91}. 
\begin{thm}\label{Kato's_thm}
For any $j \in \Z_{\geq 1}$ and any intermediate field $\Q \subset L \subset K_{\infty}$ such that 
$[L : \Q] <\infty$ and $L$ is a CM-field, the following composite map 
\begin{align*}
&
\varepsilon_0 \parenth{ \varprojlim_n \Q_p \otimes_{\Z_p} \Xi_{K_{n} / \Q, S} (0) }
\overset{ \Twist_{1, 1-j}^{\Xi}}{\xrightarrow{\sim}} 
\varepsilon_j\parenth{ \varprojlim_n \Q_p \otimes_{\Z_p} \Xi_{K_{n} / \Q, S} (j) } 
\\
&\twoheadrightarrow 
\varepsilon_j (\Q_p \otimes_{\Z_p} \Xi_{L/ \Q, S} (j) )
\overset{\vartheta_{L / \Q, S}^j}{\hookrightarrow} 
\varepsilon_j \C_p [\Gal (L / \Q)]
\end{align*}
sends the element $\varepsilon_0 Z_{N'p^\infty}^{\cyc}$ to 
\[
(-1)^{j-1} \cdot \varepsilon_j \Theta_{L/ \Q , S} (j)^\# .
\]
\end{thm}

From this and Theorem \ref{thm:zeta>0}, we know the following result. 
\begin{cor}\label{cor:k=Q}
Let $S$ be the set of places of $\Q$ which divides $N'p\infty$. 
Then, the $\varepsilon_j \Lambda$-basis 
\[
\varepsilon_j Z_{K_\infty / \Q , S}^j \in \varepsilon_j \Xi_{K_\infty / \Q , S} (j)
\]
in Theorem \ref{thm:zeta>0} coincides with $ \Twist_{1, 1-j}^{\Xi} (\varepsilon_0 Z_{N'p^\infty}^{\cyc})$ in 
\[
\varepsilon_j\parenth{ \varprojlim_n \Q_p \otimes_{\Z_p} \Xi_{K_{n} / \Q, S} (j) } 
(\supset \varepsilon_j \Xi_{K_\infty / \Q , S} (j))
\]
for any $j \in \Z_{\geq 1}$. 
\end{cor}

%%%%%%%%%%%%%%%%%%%%%%%%%%%%%%%%%%%%%%%%%%%%%%%%%%%%%%%%%
\subsection{Explicit description of $\eta_{L, S}$ when $k =\Q$}
%%%%%%%%%%%%%%%%%%%%%%%%%%%%%%%%%%%%%%%%%%%%%%%%%%%%%%%%%
In this subsection, we give an explicit relation between the element  
$\eta_{L, S}^j$ 
%\in \varepsilon_j \bigcap_{\Z_p[\Gal (L/\Q)]}^1 H^1 (G_{L, S} , \Z_p (1-j)) 
%= H^1 (G_{L, S} , \Z_p (1-j)) 
in Definition \ref{def:eta} and the cyclotomic unit 
for any imaginary abelian field $L$ when $k = \Q$. 

Let $L / \Q$ be a finite imaginary abelian extension with conductor $N > 1$. 
We decompose $N = N' p^a$, where $N'$ is a positive integer which is prime to $p$ and $a \geq 0$. 
Set $K := \Q (\mu_{N'})$ and keep the other notation in the previous subsection. 

We fix an embedding $\iota_{K_\infty} : K_\infty \hookrightarrow \C$. 
For any intermediate field $\Q \subset M \subset K_\infty$, 
we set $\iota_{M} := \iota_{K_\infty}|_{M} : M \hookrightarrow \C$. 
Since $L$ is contained in $K_\infty$, 
we obtain a fixed embedding $\iota_L : L \hookrightarrow \C$. 
This determines the element 
\[
\eta_{L, S}^j 
\in \varepsilon_j \bigcap_{\Z_p[\Gal (L/\Q)]}^1 H^1 (G_{L, S} , \Z_p (1-j)) 
= \varepsilon_jH^1 (G_{L, S} , \Z_p (1-j)) 
\]
in Definition \ref{def:eta} for any $j \in \Z_{\geq 1}$ 
(recall that the Definition of $\eta_{L, S}^j $ depends on the choice of this embedding). 

On the other hand, we define an element 
$
c_L^j \in \varepsilon_jH^1 (G_{L, S} , \Z_p (1-j)) 
$
as follows. 
For any $n \in \Z_{\geq 1}$ and any $j \in \Z_{\geq 1}$, we consider an element  
\begin{align*}
\iota_{K_n}^{-1} (c_{N'p^n}) \otimes \iota_{K_n}^{-1} (\alpha_{N'p^n})^{\otimes -j} 
&\in H^1 (G_{K_n, S} , \Z / p^n (1)) \otimes \Z / p^n (-j)
\\
&=  H^1 (G_{K_n, S} , \Z / p^n (1-j)).  
\end{align*}
If $n \geq a$, we have $L \subset K_n$. 
Thus, we can define an element 
\[
c_{L, n}^j := \cores_{K_n / L} \parenth{ 
\iota_{K_n}^{-1} (c_{N'p^n}) \otimes \iota_{K_n}^{-1} (\alpha_{N'p^n})^{\otimes -j} } 
\in  H^1 (G_{L, S} , \Z / p^n (1-j))
\]
for any $n \geq a$. These elements are compatible with respect to natural maps $\Z / p^n \twoheadrightarrow \Z / p^{n-1} $. 
Thus, we can define the element 
\[
c_L^j := (c_{L, n}^j )_n \in H^1 (G_{L, S} , \Z_p (1-j)) = \varprojlim_n 
H^1 (G_{L, S} , \Z / p^n (1-j))
\]
for any $j \in \Z_{\geq 1}$. 
Since the complex conjugation $c \in \Gal (L / \Q)$ acts on $c_L^j$ by $(-1)^j$, 
we know that $c_L^j  \in \varepsilon_jH^1 (G_{L, S} , \Z_p (1-j))$. 

Recall that $S$ is the set of primes of $\Q$ which divides $Np\infty$. 
\begin{thm}\label{thm:eta=cyc}
For any $j \in \Z_{\geq 1}$, 
we have 
\[
\eta_{L, S}^j = \frac{1}{2}c_{L}^j
\] 
in $\varepsilon_jH^1 (G_{L, S} , \Z_p (1-j))$. 
\end{thm}
\begin{proof}
For any $n \geq 1$, 
we put $\zeta_{p^n} := \iota_{K_\infty}^{-1} (\exp (\frac{2 \pi \sqrt{-1}}{p^n})) \in K_\infty$. 
This determines a $\Z_p$-basis $\xi := (\zeta_{p^n})_n \in \Z_p (1)$. 
Then, the twist map $X_{K_\infty} (0) \xrightarrow{\sim} X_{K_\infty} (-j)$ defined in \eqref{twist_X}
in \S \ref{ss:twist_def} sends the $\Lambda$-basis 
$(e_{\iota_{K_n}}^0)_n \in X_{K_\infty} (0)$ to $(e_{\iota_{K_n}}^{-j})_n \in X_{K_\infty} (-j)$. 
In addition, the twist map between Iwasawa cohomology groups
\[
 \varprojlim_{n} H^1 (G_{K_n , S} , \Z_p (1)) \xrightarrow{\sim} 
\varprojlim_{n} H^1 (G_{K_n , S} , \Z_p (1-j))
\]
which is induced by \eqref{global_twist} sends the element 
$(\frac{1}{2} \cdot \iota_{K_n}^{-1} (c_{N'p^n}) )_n$ to 
$(\frac{1}{2} \cdot c_{K_n}^j )_n$. 
Therefore, we know that the following composite map 
\begin{align*}
&
\varepsilon_0 \parenth{ \varprojlim_n \Q_p \otimes_{\Z_p} \Xi_{K_{n} / \Q, S} (0) }
\overset{ \Twist_{1, 1-j}^{\Xi}}{\xrightarrow{\sim}} 
\varepsilon_j\parenth{ \varprojlim_n \Q_p \otimes_{\Z_p} \Xi_{K_{n} / \Q, S} (j) } 
\twoheadrightarrow 
\varepsilon_j (\Q_p \otimes_{\Z_p} \Xi_{L/ \Q, S} (j) ) 
\\
&= 
 \varepsilon_j H^1 (G_{L, S} , \Q_p (1-j)) \otimes_{\Z_p[\Gal (L/\Q)]} X_L (-j)^{-1}
\xrightarrow{\sim}
\varepsilon_jH^1 (G_{L, S} , \Q_p (1-j))
\end{align*}
sends the element $\varepsilon_0 Z_{N'p^\infty}^{\cyc}$ to 
$\frac{1}{2} c_L^j$ for any positive integer $j$, 
where the last isomorphism in the above composite map is given by 
$x \otimes e_{\iota_{L}}^{-j, \ast} \mapsto x$. 
Therefore, by Definition \ref{def:eta} and Corollary \ref{cor:k=Q}, 
we know that $\eta_{L, S}^j = \frac{1}{2}c_{L}^j$ in $\varepsilon_jH^1 (G_{L, S} , \Q_p (1-j))$. 
Since $\varepsilon_jH^1 (G_{L, S} , \Z_p (1-j))$ is $\Z_p$-torsion free by Lemma \ref{lem:H^1torsion-free}, 
we obtain the claim. 
\end{proof}

%%%%%%%%%%%%%%%%%%%%%%%%%%%%%%%%%%%%%%%%%%%%%%%%%%%%%%%%%
\appendix
%%%%%%%%%%%%%%%%%%%%%%%%%%%%%%%%%%%%%%%%%%%%%%%%%%%%%%%%%

%%%%%%%%%%%%%%%%%%%%%%%%%%%%%%%%%%%%%%%%%%%%%%%%%%%%%%%%%%%
\section{Twist map}\label{App:twist}
%%%%%%%%%%%%%%%%%%%%%%%%%%%%%%%%%%%%%%%%%%%%%%%%%%%%%%%%%%%

In this section, we review a twist map between the Galois cohomology complexes. 
%This kind of material should be known to experts, but we have not found comprehensive references. 

\subsection{Definition of the twists}\label{ss:twist_def}
We fix an odd prime $p$. 
Let $K/k$ be a finite abelian extension of number fields and 
$S$ a finite set of places of $k$ such that $S$ contains $S_p \cup S_\infty \cup S_{\ram} (K/k)$. 
For any $n \in \Z_{\geq 1}$, put $K_n := K(\mu_{p^n})$ and $K_\infty := \cup_{n \geq 1} K_n$. 
We consider the Iwasawa algebra $\Lambda = \Z_p[[\Gal(K_{\infty}/k)]]$. 
For any integer $m$, we write 
\[
\twist_m : \Lambda \xrightarrow{\sim} \Lambda ; 
\quad
\sigma \mapsto \chi_{\cyc }^m (\sigma)\sigma
\]
 for the ring automorphism induced by the 
$m$-th power of the cyclotomic character. 

Let $\xi = (\zeta_{p^n})_n$ be a basis of $\Z_p(1) = \varprojlim_n \mu_{p^n}(K_{\infty})$. 
We consider the Iwasawa cohomology complex $\RG_{\Iw} (G_{K_\infty, S} , \Z_p (j))$
which is defined by the projective limit of $\RG (G_{K_n, S} , \Z_p (j))$ for any $j \in \Z$. 
Then, for any integers $j$ and $j'$, we define the following $\twist_{j-j'}$-semilinear isomorphism
\begin{align}\label{global_twist}
&\RG_{\Iw} (G_{K_\infty, S} , \Z_p (j)) \simeq
\varprojlim_n \RG (G_{K_n, S} , \Z/ p^n  (j)) 
\\
&\overset{(\cup \zeta_{p^n}^{\otimes j'-j})_n}{\xrightarrow{\sim}} 
\varprojlim_n \RG (G_{K_n, S} , \Z / p^n (j') )
\simeq 
\RG_{\Iw} (G_{K_\infty, S} , \Z_p (j')), 
\end{align}
where the first and the last isomorphisms are natural ones and 
the second isomorphism is induced by the cup product with $\zeta_{p^n}^{\otimes j'-j} \in \Z / p^n (j'-j)$. 
By \cite[Lemma 7.2]{ADK}, this induces a $\twist_{j-j'}$-semilinear isomorphism between the determinant moules
\[
\Twist_{j, j'}^{\xi} : 
\Det_{\Lambda}^{-1}
(\RG_{\Iw} (G_{K_\infty, S} , \Z_p (j)))
\xrightarrow{\sim} 
\Det_{\Lambda}^{-1} (\RG_{\Iw} (G_{K_\infty, S} , \Z_p (j'))). 
\]
Clearly, this depends on the choice of $\xi \in \Z_p(1) = \varprojlim_n \mu_{p^n}(K_{\infty})$. 

Similarly, we define a twist map for the local Galois cohomology complexes. 
Let $v$ be a finite prime of $k$. 
For any $j \in \Z$, we consider a semi-local Iwasawa cohomology complex $\RG_{\Iw} (K_{\infty, v} , \Z_p (j) )$ 
which is defined by the projective limit of $\bigoplus_{w_n \mid v }\RG (K_{n, w_n} , \Z_p (j))$, 
where $w_n$ runs over the primes of $K_n$ above $v$. 
In addition, for any $j, j' \in \Z$, we define the following $\twist_{j-j'}$-semilinear isomorphism
\begin{align}\label{local_twist}
&
\RG_{\Iw} (K_{\infty, v} , \Z_p (j) )
\simeq
\varprojlim_n \bigoplus_{w_n \mid v }\RG (K_{n, w_n} , \Z/ p^n  (j)) 
\\
&
\overset{( \oplus_{w_n \mid v} \cup \zeta_{p^n}^{\otimes j'-j})_n}{\xrightarrow{\sim}} 
\varprojlim_n \bigoplus_{w_n \mid v }\RG (K_{n, w_n} , \Z/ p^n  (j')) 
\simeq
\RG_{\Iw} (K_{\infty, v} , \Z_p (j') ), 
\end{align}
where we can regard $\zeta_{p^n} \in K_{n, w_n}$ by the natural inclusion $K_n \hookrightarrow K_{n, w_n}$. 
This also induces a $\twist_{j-j'}$-semilinear isomorphism between the determinant modules
\[
\Twist_{v, j, j'}^{\xi} : 
\Det_{\Lambda}^{-1}
(\RG_{\Iw} (K_{\infty, v} , \Z_p (j) ) )
\xrightarrow{\sim} 
\Det_{\Lambda}^{-1} (\RG_{\Iw} (K_{\infty, v} , \Z_p (j') )). 
\]

Set $S_f := S \setminus S_\infty$. 
For any $j \in \Z$, we define a complex $\Delta_{K_\infty, S} (j)$ by 
\[
\Delta_{K_\infty, S} (j) := \Cone \Big{(} \RG_{\Iw} (G_{K_\infty, S} , \Z_p (j)) \to \bigoplus_{v \in S_f} \RG_{\Iw} (K_{\infty, v} , \Z_p (j) \Big{)} [-1]. 
\]
Then, for any integers $j$ and $j'$, we obtain a $\twist_{j-j'}$-semilinear isomorphism 
$
\Delta_{K_\infty, S} (j) \xrightarrow{\sim} \Delta_{K_\infty, S} (j')
$
such that the following diagram is commutative:
\begin{equation}\label{triangle_twist}
\xymatrix{
\Delta_{K_\infty , S} (j) 
\ar[r]
\ar[d]_-{\sim}
&
\RG_{\Iw} (G_{K_\infty, S} , \Z_p (j))
\ar[r]
\ar[d]^-{\eqref{global_twist}}_-{\sim}
&
\bigoplus_{v \in S_f} \RG_{\Iw} (K_{\infty, v} , \Z_p (j) )
\ar[d]^-{\eqref{local_twist}}_-{\sim}
\\
\Delta_{K_\infty , S} (j') 
\ar[r]
&
\RG_{\Iw} (G_{K_\infty, S} , \Z_p (j'))
\ar[r]
&
\bigoplus_{v \in S_f} \RG_{\Iw} (K_{\infty, v} , \Z_p (j') ) 
}
\end{equation}
and this induces a $\twist_{j-j'}$-semilinear isomorphism 
\[
\Twist_{j, j'}^{\Delta, \xi} : 
\Det_{\Lambda}^{-1}
(\Delta_{K_\infty , S} (j))
\xrightarrow{\sim} 
\Det_{\Lambda}^{-1} (\Delta_{K_\infty , S} (j') 
). 
\]

For any $j \in \Z$ and $n \in \Z_{\geq 0}$, 
we set $X_{K_n} (j) := \bigoplus_{\iota : K_n \hookrightarrow \C}\Z_p (j)$, 
where $\iota$ runs over all embeddings. 
For any $j \in \Z$, we define a $\Lambda$-module $X_{K_\infty} (j)$ by
\[
X_{K_\infty} (j) := \varprojlim_{n} X_{K_n} (j) 
\]
and consider an isomorphism
\begin{equation}\label{twist_X}
 X_{K_\infty} (j) \simeq 
\varprojlim_n \parenth{\bigoplus_{\iota : K_n \hookrightarrow \C} \Z / p^n \Z (j)}
\xrightarrow{\sim} 
\varprojlim_n \parenth{\bigoplus_{\iota : K_n \hookrightarrow \C} \Z / p^n \Z (j')} \simeq 
X_{K_\infty} (j'),
\end{equation}
which is defined by
\[
((a_{n, \iota})_{\iota : K_n \hookrightarrow \C} )_n \mapsto
 ((a_{n, \iota} \otimes \iota (\zeta_{p^n})^{\otimes j'-j})_{\iota : K_n \hookrightarrow \C} )_n.
 \]
Then, this is also $\twist_{j-j'}$-semilinear isomorphism (see \cite[\S 7.2]{ADK}). 
Therefore, this induces a $\twist_{j-j'}$-semilinear isomorphism between determinant modules 
\begin{equation}\label{shift X}
\Det_{\Lambda}^{-1} (X_{K_\infty} (j))
\xrightarrow{\sim}
\Det_{\Lambda}^{-1} (X_{K_\infty} (j')). 
\end{equation}

For any $j \in \Z$, we put 
\[
\Xi_{K_\infty / k, S}^{\loc} (j) := \bigotimes_{v \in S_f} 
  \Det_{\Lambda}^{-1} (\RG_{\Iw} (K_{\infty, v} , \Z_p (j) )
) \otimes_{\Lambda} 
  \Det_{\Lambda}^{-1} (X_{K_\infty} (j)). 
\]
Then, for any integers $j$ and $j'$, 
we define a $\twist_{j-j'}$-semilinear isomorphism 
\begin{equation}\label{loc_Twist}
\Twist_{j, j'}^{\loc} :
\Xi_{K_\infty / k, S}^{\loc} (j) 
\xrightarrow{\sim}
\Xi_{K_\infty / k, S}^{\loc} (j')
\end{equation}
as the tensor product of \eqref{shift X} and $\Twist_{v, j, j'}^{\xi}$ for each finite prime $v$ of $k$. 
Then, $\Twist_{j, j'}^{\loc}$ is independent of the choice of $\xi$ (see \cite[\S 7.2]{ADK}).

%%%%%%%%%%%%%%%%%%%%%%%%%%%%%%%%%%%%%%%%%%%%%%%%%%%%%%%%%%%
\subsection{CM-extension case}\label{ss:twist_CM}
%%%%%%%%%%%%%%%%%%%%%%%%%%%%%%%%%%%%%%%%%%%%%%%%%%%%%%%%%%%

Let $K/k$ be a finite abelian CM extension, and we use the same notation as in \S \ref{ss:main_conj}. 
For any integer $j$, we consider an idempotent 
\[
\varepsilon_j := \frac{1+(-1)^j c}{2} \in \Lambda, 
\]
where $c \in \Gal (K_\infty /k)$ is the complex conjugation. 
Then, by Lemma \ref{lem:complex} (1), (2), we know that the grade of 
$\varepsilon_{1-j} \Det_{\Lambda}^{-1} (\Delta_{K_\infty , S} (j))$ is zero for any $ j \in \Z_{\leq 0}$.  
Thus, the $\varepsilon_{1-j}$-component of the isomorphism $\Twist_{j, j'}^{\Delta, \xi}$ 
does not depend on the choice of $\xi$ for any non-positive integers $j$ and $ j'$ (this can be checked as \cite[Lemma 7.3]{ADK}). 
Therefore, for simplicity, we write 
\begin{equation}\label{twist_Delta}
\Twist_{j, j'}^{\Delta} : 
\varepsilon_{1-j} \Det_{\Lambda}^{-1}
(\Delta_{K_\infty , S} (j))
\xrightarrow{\sim} 
\varepsilon_{1-j} \Det_{\Lambda}^{-1} (\Delta_{K_\infty , S} (j') 
)
\end{equation}
for this twist map. 

For any $j \in \Z_{\geq 1}$, put $X_{K_\infty} (-j)^+ := \varepsilon_jX_{K_\infty} (-j)$
and define 
\[
\Xi_{K_\infty, k , S} (j) := \Det_{\Lambda}^{-1}(\RG_{\Iw} (G_{K_\infty, S} , \Z_p (1-j))) 
\otimes_{\Lambda} 
 \Det_{\Lambda}^{-1} (X_{K_\infty} (-j)^+). 
\]
Then, by Lemma \ref{lem:complex} (3), 
we know that the Euler characteristic of $\varepsilon_j \RG_{\Iw} (G_{K_\infty, S} , \Z_p (1-j))$ as 
a perfect complex of $\varepsilon_j\Lambda$-modules is $-[k:\Q]$ and 
the $\varepsilon_j \Lambda$-module $\varepsilon_jX_{K_\infty} (-j)^+$ is free of rank $[k :\Q]$. 
Thus, the grade of $\Xi_{K_\infty, k , S} (j)$ is zero. 
Therefore, for any positive integers $j$ and $j'$, 
the following $\twist_{j'-j}$-semilinear isomorphism
\[
\Twist_{1-j, 1-j'}^{\Xi}%:= \Twist_{1-j, 1-j'}^{\xi} \otimes  \eqref{shift X}
: \varepsilon_j\Xi_{K_\infty, k , S} (j) 
\xrightarrow{\sim}
\varepsilon_j\Xi_{K_\infty, k , S} (j')
\]
defined by the tensor product of $\Twist_{1-j, 1-j'}^{\xi}$ with  \eqref{shift X}
is also independent of the choice of $\xi$.

{
\bibliographystyle{abbrv}
\bibliography{biblio}
}

\end{document}